\documentclass[10pt]{amsart}
\usepackage[T1]{fontenc}
\usepackage[utf8]{inputenc}
\usepackage[english]{babel}
\usepackage{lmodern}
\usepackage{amsmath,amssymb,mathtools}
\usepackage{microtype}
\usepackage{needspace}
\usepackage[unicode,hidelinks,bookmarksnumbered]{hyperref}
\usepackage{bookmark}
\allowdisplaybreaks[1]
\theoremstyle{plain}
\newtheorem{theorem}{Theorem}[section]
\newtheorem*{introtheorem}{Theorem A}
\newtheorem{proposition}[theorem]{Proposition}
\newtheorem{corollary}[theorem]{Corollary}
\newtheorem{lemma}[theorem]{Lemma}
\theoremstyle{definition}
\newtheorem{definition}[theorem]{Definition}
\newtheorem{example}[theorem]{Example}
\numberwithin{equation}{section}

\title[Invariant measures over higher local fields]{Invariant Measures on Smooth Algebraic Groups over Higher Local Fields}
\author{Xuecai Ma}
\address{Megawave Fusion, Hefei, Anhui 230071, P.R. China}
\email{xuecai.u.ma@foxmail.com}
\author{Jingwen Zhu}
\address{Westlake University, Hangzhou, Zhejiang 310030, P.R. China}
\email{zhujingwen@westlake.edu.cn}
\date{}

\begin{document}
\begin{abstract}
We construct a positive, finitely additive, left-invariant measure on the congruence cylinder ring of a smooth $d$-dimensional algebraic group over an $n$-dimensional local field. The measure takes values in $\mathbb{R}((X_2))\cdots((X_n))$. A normalized cylinder of depth $(i_1,\ldots,i_n)$ has volume $q^{-di_1}\prod_{\ell=2}^nX_\ell^{di_\ell}$, where $q$ is the final residue-field cardinality. We also obtain coordinate-change and quotient integration formulas for this measure.
\end{abstract}

\maketitle
\tableofcontents

\section{Introduction}\label{sec:intro}

Higher local fields arise from flags on arithmetic schemes and have successive residue fields ending in a finite field \cite{zhukov,morrow-higher}. Local parameters combine the successive valuations into a valuation with several components, recording the arithmetic data at each residue level. Over a one-dimensional local field, normalized volumes of compact open additive subgroups are determined by finite coset indices \cite[Chapter~I]{weil}. In higher dimension, the outer valuation topology is not locally compact, and one valuation ideal may have infinite index in another. Measures must therefore record higher valuation levels while respecting finite disjoint decompositions.

Measures with values in iterated Laurent series fields record finite indices in real coefficients and higher valuation levels in formal variables. Fesenko~\cite{fesenko,fesenko-higher} constructed such measures on additive groups of higher local fields. Morrow developed integration on valuation fields, product spaces, and general linear groups \cite{morrow-product,morrow-valuations}, and studied nonlinear changes of variables \cite{morrow-fubini}. Waller~\cite{waller-gl2} treated congruence measures on $\operatorname{GL}_2$ over two-dimensional local fields; Mori~\cite{mori} developed measure theory and integration on $\operatorname{GL}_m$ and its parabolic subgroups over higher local fields. Waller's level structures and groups of ordered type provide a general framework, with applications to reductive algebraic groups \cite{waller-levels,waller}.

Related approaches to harmonic analysis and representations use Hecke algebras to encode convolution of functions on suitable double cosets \cite{kapranov,kim-lee,lee}. Pro-vector spaces provide a categorical framework for representations of algebraic groups over two-dimensional local fields \cite{gaitsgory-kazhdan}, while filtered adelic spaces support Fourier transforms and Poisson formulas for algebraic surfaces over finite fields \cite{osipov-parshin}.

In the present paper, we construct a measure on congruence cylinders for every smooth algebraic group. In analytic coordinates at the identity, the nonlinear terms in the group law have higher valuation, so the preimages of sufficiently deep scalar balls are subgroups. For groups of positive dimension, increasing the innermost valuation component gives a subgroup of finite index, whereas increasing a higher component gives infinite index. These two cases are recorded by the real and formal factors in the volume formula. A finite partition argument extends the cylinder volumes to a measure that is positive on every nonempty set in the cylinder ring.

Jacobian determinants give the change-of-coordinates formula on small neighborhoods. Weighting the coordinate measure by a left-invariant top-degree differential form makes these measures agree on the corresponding local rings. We then compute the effect of right translation and prove a quotient integration formula. The final section gives examples, including matrix groups, tori, and abelian varieties over fields of equal and mixed characteristic. For smooth closed matrix subgroups, Proposition~\ref{prop:matrix} identifies sufficiently deep matrix congruence subgroups with scalar balls in suitable identity coordinates.

Fix $n\ge2$. Write $F=F_n$ for an $n$-dimensional local field, with successive residue fields $F_{n-1},\ldots,F_0=k$ and $|k|=q$. Local parameters give a valuation $v:F^\times\to\mathbb{Z}^n$; set $v(0)=+\infty$. Compare tuples from the right, so the outer component is compared first. For integers $i_1,\ldots,i_n$, put
\[
\mathfrak{p}_{i_1,\ldots,i_n}=\{a\in F:v(a)\ge(i_1,\ldots,i_n)\}.
\]
For formal variables $X_2,\ldots,X_n$, let $\mathbb R_n=\mathbb{R}((X_2))\cdots((X_n))$, ordered recursively by the sign of the lowest nonzero coefficient, starting with $X_n$. Normalized depths range over
\[
\mathcal I=\{(i_1,\ldots,i_n)\in\mathbb{Z}^n:(i_1,\ldots,i_n)\ge(0,\ldots,0)\},
\]
with the same right lexicographic order. For a subset $S\subseteq F$, write $S^d$ for its $d$-fold Cartesian product. For a family of sets $\mathcal C$, $\operatorname{ring}\mathcal C$ denotes its closure under finite unions and differences.

\begin{introtheorem}[Theorem~\ref{thm:main}]
Let $G/F$ be a smooth algebraic group of dimension $d$. If $d>0$, choose an analytic chart $z:U\xrightarrow{\sim}V\subset F^d$ on an open neighborhood of the identity $e$, with $z(e)=0$. There are integers $\kappa_1,\ldots,\kappa_{n-1}\ge0$ and $\kappa_n\ge1$ such that all balls $\mathfrak{p}_{i_1+\kappa_1,\ldots,i_n+\kappa_n}^{\,d}$ with $(i_1,\ldots,i_n)\in\mathcal I$ lie in $V$, and
\[
H^z_{i_1,\ldots,i_n}:=z^{-1}(\mathfrak{p}_{i_1+\kappa_1,\ldots,i_n+\kappa_n}^{\,d}),
\qquad (i_1,\ldots,i_n)\in\mathcal I,
\]
are subgroups. Their left cosets are called \emph{congruence cylinders}. The ring
\[
\mathcal{R}_z(G(F))=\operatorname{ring}\{gH^z_{i_1,\ldots,i_n}:g\in G(F),\ (i_1,\ldots,i_n)\in\mathcal I\}
\]
has a unique finitely additive left-invariant measure $\mu^0_{G,z}:\mathcal{R}_z(G(F))\to\mathbb R_n$ with
\[
\boxed{\ \mu^0_{G,z}(gH^z_{i_1,\ldots,i_n})=q^{-di_1}\prod_{\ell=2}^nX_\ell^{di_\ell}.\ }
\]
Every nonempty measurable set has strictly positive measure. For $d=0$, counting measure on finite subsets has these properties.
\end{introtheorem}

Lemma~\ref{lem:image} shows that an analytic local diffeomorphism $f$ maps sufficiently small scalar balls according to
\[
f(a+\mathfrak{p}_{i_1,\ldots,i_n}^{\,d})
=f(a)+f'(a)\mathfrak{p}_{i_1,\ldots,i_n}^{\,d},
\]
where $f'(a)$ is the Jacobian matrix at $a$. This supplies the local change-of-variables rule for volumes.

For an analytic homomorphism $\pi:A\to Q$ of finite-dimensional $F$-analytic groups with surjective differential at the identity, put $N=\ker\pi$. Proposition~\ref{prop:quotient} constructs compatible deep congruence subgroups forming exact sequences, with multiplicative finite inner indices. With compatibly normalized cylinder measures $\mu_N,\mu_A,\mu_Q$, Proposition~\ref{prop:disintegration} gives
\[
\int_A\varphi\,d\mu_A
=\int_Q\left(\int_N\varphi(a_yh)\,d\mu_N(h)\right)d\mu_Q(y).
\]
Here $\varphi$ is a finite linear combination of indicators of the compatible congruence cylinders. For $y\in\pi(A)$, the inner integral is independent of the lift $a_y$ satisfying $\pi(a_y)=y$; it is taken to be zero outside $\pi(A)$.

We give an outline of this article. In Section~\ref{sec:construction}, we prove the main theorem, from the construction of deep congruence subgroups to the extension of their volumes.  In Section~\ref{sec:properties}  we study normalization, coordinate and parameter changes, and the behavior of the measure under translations, products, and quotients. Section~\ref{sec:examples} contains the examples, the results on matrix congruences, and explicit volume calculations.

\section{Construction of the invariant measure}\label{sec:construction}

For $1\le\ell\le n$, the field $F_\ell$ is complete for a discrete valuation $\operatorname{ord}_\ell$ with residue field $F_{\ell-1}$ \cite{zhukov}. We let $\mathcal{O}_{F_\ell}$ denote its discrete valuation ring and choose a uniformizer $\pi_\ell\in F_\ell$. We can lift these parameters successively to elements $t_1,\ldots,t_n\in F$, with $t_n=\pi_n$. Put $\mathcal{O}_F=\mathcal{O}_{F_n}$ and $\mathfrak m_F=t_n\mathcal{O}_F$. A bar at stage $\ell$ denotes reduction $\mathcal{O}_{F_\ell}\to F_{\ell-1}$.

For $a\ne0$, set $a_n=a$ and recursively define
\begin{equation}\label{eq:valuation}
i_\ell=\operatorname{ord}_\ell(a_\ell),\qquad
a_{\ell-1}=\overline{a_\ell\pi_\ell^{-i_\ell}}
\quad(\ell=n,\ldots,1),\qquad
v(a)=(i_1,\ldots,i_n).
\end{equation}
Write $v_\ell(a)$ for the $\ell$th component of $v(a)$, with $v_\ell(0)=+\infty$. The $\ell$th component of $v(t_\ell)$ is $1$, and all its other components are $0$. The rank-$n$ valuation ring and its principal fractional ideals are
\[
\mathcal{O}=\mathcal{O}_F^{(n)}=\{a:v(a)\ge(0,\ldots,0)\},\qquad
\mathfrak{p}_{i_1,\ldots,i_n}=t_1^{i_1}\cdots t_n^{i_n}\mathcal{O}.
\]
The same recursion defines $v:F_\ell^\times\to\mathbb{Z}^\ell$ on each intermediate residue field. When the field is indicated explicitly, we write
\[
\mathfrak{p}_{i_1,\ldots,i_\ell}(F_\ell)
=\{a\in F_\ell:v(a)\ge(i_1,\ldots,i_\ell)\},
\qquad v(0)=+\infty.
\]
Successive reduction gives a surjection $\rho_1:\mathcal{O}\to\mathcal{O}_{F_1}$ with kernel
\[
\mathfrak q=\{0\}\cup\{a\in\mathcal{O}\setminus\{0\}:(v_2(a),\ldots,v_n(a))>(0,\ldots,0)\}.
\]

\begin{definition}\label{def:vector-valuation}
For $d\ge1$ and $x=(x_1,\ldots,x_d)\in F^d$, define
\[
v(x)=\min_{1\le s\le d}v(x_s)\in\mathbb{Z}^n\cup\{+\infty\},
\qquad v(0)=+\infty.
\]
The minimum is taken in the right lexicographic order.
\end{definition}

\begin{proposition}\label{prop:ideal}
The valuation ideals satisfy $\mathfrak{p}_{i_1,\ldots,i_n}\mathfrak{p}_{j_1,\ldots,j_n}\subseteq\mathfrak{p}_{i_1+j_1,\ldots,i_n+j_n}$. For $(j_1,\ldots,j_n)\ge(i_1,\ldots,i_n)$,
\begin{equation}\label{eq:idealindex}
[\mathfrak{p}_{i_1,\ldots,i_n}:\mathfrak{p}_{j_1,\ldots,j_n}]=
\begin{cases}
q^{j_1-i_1},&(j_2,\ldots,j_n)=(i_2,\ldots,i_n),\\
\infty,&(j_2,\ldots,j_n)>(i_2,\ldots,i_n).
\end{cases}
\end{equation}
In particular, $\mathfrak{p}_{i_1,\ldots,i_n}/\mathfrak{p}_{i_1+1,i_2,\ldots,i_n}\simeq(k,+)$ and $[\mathfrak{p}_{i_1,\ldots,i_n}:\mathfrak{p}_{i_1+m,i_2,\ldots,i_n}]=q^m$ for $m\ge0$.
\end{proposition}

\begin{proof}
For $a\in\mathfrak{p}_{i_1,\ldots,i_n}$ and $b\in\mathfrak{p}_{j_1,\ldots,j_n}$,
\[
v(ab)=v(a)+v(b)\ge(i_1+j_1,\ldots,i_n+j_n).
\]
At a fixed level $(i_2,\ldots,i_n)$, multiplication by $(t_1^{i_1}\cdots t_n^{i_n})^{-1}$ and successive reduction give
\[
\mathfrak{p}_{i_1,\ldots,i_n}/\mathfrak{p}_{i_1+m,i_2,\ldots,i_n}
\simeq\mathcal{O}/t_1^m\mathcal{O}
\simeq\mathcal{O}_{F_1}/\pi_1^m\mathcal{O}_{F_1}.
\]
The last quotient has $q^m$ elements. At a strictly larger level,
\[
\begin{aligned}
(j_2,\ldots,j_n)>(i_2,\ldots,i_n)
&\ \Longrightarrow\
\mathfrak{p}_{j_1,\ldots,j_n}\subseteq\mathfrak{p}_{i_1+m,i_2,\ldots,i_n}
\quad(m\ge0),\\
[\mathfrak{p}_{i_1,\ldots,i_n}:\mathfrak{p}_{j_1,\ldots,j_n}]
&\ge q^m\quad(m\ge0)
\ \Longrightarrow\
[\mathfrak{p}_{i_1,\ldots,i_n}:\mathfrak{p}_{j_1,\ldots,j_n}]=\infty.
\end{aligned}
\]
\end{proof}

Put $\Lambda_n=\mathbb{C}((X_2))\cdots((X_n))$. The order on $\mathbb R_n$ makes each $X_\ell$ positive and smaller than every positive element of $\mathbb{R}((X_2))\cdots((X_{\ell-1}))$. For $0\ne c\in\Lambda_n$, write $\operatorname{lc}(c)\in\mathbb{C}$ for the coefficient of the least exponent tuple $(i_2,\ldots,i_n)$ in the right lexicographic order. The same convention applies to real coefficients and other formal variables; a nonzero real series is positive precisely when its leading coefficient is positive.

For $a\in F^\times$ with $v(a)=(i_1,\ldots,i_n)$, we define
\begin{equation}\label{eq:modulus}
\left|a\right|_n=q^{-i_1}\prod_{\ell=2}^nX_\ell^{i_\ell},\qquad
\left|L\right|_n=\left|\det L\right|_n\quad(L\in\operatorname{GL}_d(F)).
\end{equation}
A \emph{scalar ball} is a translate of $\mathfrak{p}_{i_1,\ldots,i_n}^{\,d}$. Its absolute coordinate volume is
\begin{equation}\label{eq:absoluteball}
\operatorname{vol}(\mathfrak{p}_{i_1,\ldots,i_n}^{\,d})=q^{-di_1}\prod_{\ell=2}^nX_\ell^{di_\ell}.
\end{equation}
A translated scalar ball has the same volume.

We use the outer discrete valuation topology for analytic convergence and the inverse and implicit function theorems over complete valued fields \cite{glockner}. The inclusions
\[
t_n^{i_n+1}\mathcal{O}_F\subseteq\mathfrak{p}_{i_1,\ldots,i_n}\subseteq t_n^{i_n}\mathcal{O}_F
\]
show that the scalar balls form a closed neighborhood basis. Their quotients by $t_n^{i_n+1}\mathcal{O}_F$ are infinite discrete groups: reduction identifies them with a fractional ideal of the rank-$(n-1)$ valuation ring of $F_{n-1}$. Thus these balls are not compact, and $F$ is not locally compact.

Let $A$ be a $d$-dimensional $F$-analytic group, with $d>0$. Choose a chart $z:U\xrightarrow{\sim}V\subset F^d$ near its identity $e$, with $z(e)=0$. Multiplication and inversion have coordinate expressions
\begin{equation}\label{eq:grouplaw}
\Phi(x,y)=x+y+R(x,y),\qquad \iota(x)=-x+S(x),
\end{equation}
where $x=(x_1,\ldots,x_d)$ and $y=(y_1,\ldots,y_d)$ belong to $F^d$, with
\[
v(x_s),v(y_s)\in\mathbb{Z}^n\cup\{+\infty\}\qquad(1\le s\le d).
\]
The coordinate maps are $\Phi(x,y)=z(z^{-1}(x)z^{-1}(y))$ and $\iota(x)=z((z^{-1}(x))^{-1})$. Every monomial of $R$ contains a variable from each input, and every monomial of $S$ has degree at least two. For $g,h\in A$, put $[g,h]=ghg^{-1}h^{-1}$. For subgroups $B,C$, let $[B,C]$ be the subgroup generated by their commutators. Each monomial in the coordinate expression of the commutator contains variables from both inputs.

Write
\[
C(x,y)=z([z^{-1}(x),z^{-1}(y)])
\]
for the coordinate expression of the commutator, defined for $x,y$ sufficiently close to $0$.

\begin{lemma}\label{lem:estimate}
There are integers $\kappa_1,\ldots,\kappa_{n-1}\ge0$ and $\kappa_n\ge1$ such that, for every pair of depths satisfying
\[
(i_1,\ldots,i_n)\ge(\kappa_1,\ldots,\kappa_n),\qquad
(j_1,\ldots,j_n)\ge(\kappa_1,\ldots,\kappa_n),
\]
the coordinate expressions are defined on the indicated balls and satisfy
\begin{equation}\label{eq:estimate}
\begin{aligned}
R(\mathfrak{p}_{i_1,\ldots,i_n}^{\,d},\mathfrak{p}_{j_1,\ldots,j_n}^{\,d})&\subseteq\mathfrak{p}_{i_1+j_1-\kappa_1,\ldots,i_n+j_n-\kappa_n+1}^{\,d},\\
S(\mathfrak{p}_{i_1,\ldots,i_n}^{\,d})&\subseteq\mathfrak{p}_{2i_1-\kappa_1,\ldots,2i_n-\kappa_n+1}^{\,d},\\
C(\mathfrak{p}_{i_1,\ldots,i_n}^{\,d},\mathfrak{p}_{j_1,\ldots,j_n}^{\,d})&\subseteq\mathfrak{p}_{i_1+j_1-\kappa_1,\ldots,i_n+j_n-\kappa_n+1}^{\,d}.
\end{aligned}
\end{equation}
\end{lemma}

\begin{proof}
All valuation inequalities use the right lexicographic order. Coordinatewise,
\[
x\in\mathfrak{p}_{\alpha_1,\ldots,\alpha_n}^{\,d}
\ \Longleftrightarrow\
v(x_s)\ge(\alpha_1,\ldots,\alpha_n)\ (1\le s\le d)
\ \Longleftrightarrow\
v(x)\ge(\alpha_1,\ldots,\alpha_n).
\]
Shrink the chart so that all three coordinate maps are analytic. The group identities give
\[
\begin{aligned}
\Phi(x,0)=x,\quad \Phi(0,y)=y
&\ \Longrightarrow\ R(x,0)=R(0,y)=0,\\
[g,e]=[e,h]=e
&\ \Longrightarrow\ C(x,0)=C(0,y)=0.
\end{aligned}
\]
For $I=(I_1,\ldots,I_d)\in\mathbb{Z}_{\ge0}^d$, write $x^I=\prod_r x_r^{I_r}$ and $|I|=\sum_r I_r$. Uniqueness of power-series coefficients yields
\[
R_s=\sum_{I,J}a_{s,I,J}x^Iy^J,\quad
C_s=\sum_{I,J}c_{s,I,J}x^Iy^J
\quad\Longrightarrow\quad
\begin{cases}
a_{s,I,0}=a_{s,0,J}=0,\\
c_{s,I,0}=c_{s,0,J}=0.
\end{cases}
\]

For inversion, $\iota(0)=0$ and $S_s(x)=\iota_s(x)+x_s$. With $\delta_{sr}$ the Kronecker delta,
\[
\frac{\partial\Phi_s}{\partial x_r}(0,0)
=\frac{\partial\Phi_s}{\partial y_r}(0,0)=\delta_{sr}.
\]
Differentiating $\Phi_s(x,\iota(x))=0$ componentwise gives
\[
\begin{aligned}
0&=\frac{\partial\Phi_s}{\partial x_r}(0,0)
+\sum_{t=1}^d\frac{\partial\Phi_s}{\partial y_t}(0,0)
 \frac{\partial\iota_t}{\partial x_r}(0)\\
&=\delta_{sr}+\frac{\partial\iota_s}{\partial x_r}(0),\\
\frac{\partial S_s}{\partial x_r}(0)
&=\frac{\partial\iota_s}{\partial x_r}(0)+\delta_{sr}=0.
\end{aligned}
\]
Thus, including the vanishing constant terms,
\[
\begin{aligned}
R_s(x,y)&=\sum_{|I|,|J|\ge1}a_{s,I,J}x^Iy^J,\\
C_s(x,y)&=\sum_{|I|,|J|\ge1}c_{s,I,J}x^Iy^J,\\
S_s(x)&=\sum_{|K|\ge2}b_{s,K}x^K
\qquad(1\le s\le d).
\end{aligned}
\]

For the coefficient estimates below, $h_\gamma$ denotes a coefficient of one of these component series, according to
\[
(h_\gamma,\gamma,m)=
\begin{cases}
(a_{s,I,J},(I,J),|I|+|J|),&\text{for }R_s,\\
(c_{s,I,J},(I,J),|I|+|J|),&\text{for }C_s,\\
(b_{s,K},K,|K|),&\text{for }S_s.
\end{cases}
\]
Thus $m=|\gamma|\ge2$ is the total degree. The component index $s$ and the choice of series are suppressed in the notation $h_\gamma$.

Choose $\rho_0\in F^\times$ so that $D_0=(\rho_0\mathcal{O}_F)^d\subset V$, the series for $R$ and $C$ converge on $D_0\times D_0$, and the series for $S$ converge on $D_0$. Taking these closed polydiscs inside the respective convergence domains gives a common coefficient bound
\[
\exists B\ge0\quad
v_n(h_\gamma\rho_0^m)\ge-B
\qquad\text{for all component coefficients}.
\]

Fix $L\ge1$ and set $\rho=t_n^L\rho_0$, $D=(\rho\mathcal{O}_F)^d\subset D_0$. Then
\[
\begin{aligned}
v_n(h_\gamma\rho^m)&\ge-B+mL,\\
v_n(h_\gamma\rho^{m-2})&\ge-B+mL-2v_n(\rho)
\xrightarrow[m\to\infty]{}+\infty.
\end{aligned}
\]
Let $\mathcal B$ consist of the nonzero members of
\[
\left\{
a_{s,I,J}\rho^{|I|+|J|-2},\
c_{s,I,J}\rho^{|I|+|J|-2},\
b_{s,K}\rho^{|K|-2}
\right\}.
\]
Choose $M_0\ge2$ with $M_0L>B+2v_n(\rho)$. The preceding bound implies
\[
\begin{aligned}
\mathcal E&:=\{\beta\in\mathcal B:v_n(\beta)\le0\},\\
\mathcal E&\subseteq\{\text{normalized coefficients of degree at most }M_0\},\\
\#\mathcal E&<\infty.
\end{aligned}
\]
Define the component constants by
\[
\begin{aligned}
\kappa_s
&=\max\bigl(\{0\}\cup\{-v_s(\beta):\beta\in\mathcal E\}\bigr)
\qquad(1\le s<n),\\
\kappa_n
&=1+\max\bigl(\{0,v_n(\rho)\}\cup\{-v_n(\beta):\beta\in\mathcal E\}\bigr).
\end{aligned}
\]
These maxima exist even when $\mathcal E=\varnothing$, and
\[
\kappa_s\ge0\quad(1\le s<n),\qquad
\kappa_n\ge1,\qquad \kappa_n>v_n(\rho).
\]
The extra $1$ in $\kappa_n$ leaves one outer valuation step in the estimates below. In particular,
\[
\begin{cases}
v(\beta)\ge(-\kappa_1,\ldots,-\kappa_{n-1},1-\kappa_n),&\beta\in\mathcal E,\\
v(\beta)>(0,\ldots,0)\ge(-\kappa_1,\ldots,-\kappa_{n-1},1-\kappa_n),
&\beta\in\mathcal B\setminus\mathcal E.
\end{cases}
\]
Since $v(0)=+\infty$, all coefficients satisfy
\[
\begin{aligned}
v(a_{s,I,J}\rho^{|I|+|J|-2})&\ge(-\kappa_1,\ldots,-\kappa_{n-1},1-\kappa_n),\\
v(c_{s,I,J}\rho^{|I|+|J|-2})&\ge(-\kappa_1,\ldots,-\kappa_{n-1},1-\kappa_n),\\
v(b_{s,K}\rho^{|K|-2})&\ge(-\kappa_1,\ldots,-\kappa_{n-1},1-\kappa_n).
\end{aligned}
\]
These constants depend only on the component series and $\rho$.

Suppose $v(x),v(y)\ge v(\rho)$ and put $u=x/\rho$, $w=y/\rho$ coordinatewise. Then
\[
\begin{aligned}
v(u_s)&=v(x_s)-v(\rho)\ge0,\qquad v(w_s)\ge0,\\
u,w&\in\mathcal{O}^d,\qquad
x,y\in(\rho\mathcal{O})^d\subseteq D.
\end{aligned}
\]
Thus the series converge, and every monomial in $u,w$ has nonnegative valuation.

For $|I|,|J|\ge1$, choose $p,r$ with $I_p,J_r\ge1$. With $e_p,e_r$ the standard multiindices,
\[
x^I=x_p\rho^{|I|-1}u^{I-e_p},\qquad
y^J=y_r\rho^{|J|-1}w^{J-e_r}.
\]
Consequently,
\[
\begin{aligned}
a_{s,I,J}x^Iy^J
&=(a_{s,I,J}\rho^{|I|+|J|-2})x_py_ru^{I-e_p}w^{J-e_r},\\
v(a_{s,I,J}x^Iy^J)
&=v(a_{s,I,J}\rho^{|I|+|J|-2})+v(x_p)+v(y_r)\\
&\qquad+v(u^{I-e_p})+v(w^{J-e_r})\\
&\ge v(x)+v(y)-(\kappa_1,\ldots,\kappa_{n-1},\kappa_n-1).
\end{aligned}
\]
The identical factorization for $c_{s,I,J}$ gives
\[
v(c_{s,I,J}x^Iy^J)\ge v(x)+v(y)-(\kappa_1,\ldots,\kappa_{n-1},\kappa_n-1).
\]

For $|K|\ge2$, choose $p,r$ so that
\[
K-e_p-e_r\in\mathbb{Z}_{\ge0}^d:
\qquad
\begin{cases}
r=p,&K_p\ge2,\\
r\ne p,\ K_p,K_r\ge1,&\text{otherwise}.
\end{cases}
\]
Then
\[
\begin{aligned}
b_{s,K}x^K
&=(b_{s,K}\rho^{|K|-2})x_px_ru^{K-e_p-e_r},\\
v(b_{s,K}x^K)
&\ge v(x_p)+v(x_r)-(\kappa_1,\ldots,\kappa_{n-1},\kappa_n-1)\\
&\ge2v(x)-(\kappa_1,\ldots,\kappa_{n-1},\kappa_n-1).
\end{aligned}
\]
The inequalities include zero terms by the convention $v(0)=+\infty$.

For the depths in the statement, $\kappa_n>v_n(\rho)$ gives
\[
\begin{aligned}
(i_1,\ldots,i_n)&\ge(\kappa_1,\ldots,\kappa_n)>v(\rho),\\
(j_1,\ldots,j_n)&\ge(\kappa_1,\ldots,\kappa_n)>v(\rho).
\end{aligned}
\]
Hence, for $x\in\mathfrak{p}_{i_1,\ldots,i_n}^{\,d}$ and $y\in\mathfrak{p}_{j_1,\ldots,j_n}^{\,d}$,
\[
\begin{aligned}
v(a_{s,I,J}x^Iy^J),\ v(c_{s,I,J}x^Iy^J)
&\ge(i_1+j_1-\kappa_1,\ldots,i_n+j_n-\kappa_n+1),\\
v(b_{s,K}x^K)&\ge(2i_1-\kappa_1,\ldots,2i_n-\kappa_n+1).
\end{aligned}
\]
The same threshold also gives
\[
\begin{aligned}
&(i_1+j_1-\kappa_1,\ldots,i_n+j_n-\kappa_n+1)\\
&\hspace{2em}>\max\{(i_1,\ldots,i_n),(j_1,\ldots,j_n)\},\\
&(2i_1-\kappa_1,\ldots,2i_n-\kappa_n+1)>(i_1,\ldots,i_n).
\end{aligned}
\]

It remains to take limits. Each valuation ideal is additive:
\[
z_\ell\in\mathfrak{p}_{\alpha_1,\ldots,\alpha_n}\ (1\le\ell\le M)
\ \Longrightarrow\
v\left(\sum_{\ell=1}^M z_\ell\right)
\ge\min_\ell v(z_\ell)\ge(\alpha_1,\ldots,\alpha_n).
\]
It is also closed in the outer topology. Indeed, for the open subgroup
$H=t_n^{\alpha_n+1}\mathcal{O}_F$,
\[
H\subseteq\mathfrak{p}_{\alpha_1,\ldots,\alpha_n},
\qquad
F\setminus\mathfrak{p}_{\alpha_1,\ldots,\alpha_n}
=\bigcup_{a\notin\mathfrak{p}_{\alpha_1,\ldots,\alpha_n}}(a+H),
\]
so the complement is open.

Let $R_{s,\le M},C_{s,\le M},S_{s,\le M}$ denote the corresponding truncations by total degree. The termwise bounds yield
\[
\begin{aligned}
R_{s,\le M}(x,y),C_{s,\le M}(x,y)
&\in\mathfrak{p}_{i_1+j_1-\kappa_1,\ldots,i_n+j_n-\kappa_n+1},\\
S_{s,\le M}(x)&\in\mathfrak{p}_{2i_1-\kappa_1,\ldots,2i_n-\kappa_n+1}.
\end{aligned}
\]
Taking outer limits in these closed ideals gives
\[
\begin{aligned}
R_s(x,y),C_s(x,y)
&\in\mathfrak{p}_{i_1+j_1-\kappa_1,\ldots,i_n+j_n-\kappa_n+1},\\
S_s(x)&\in\mathfrak{p}_{2i_1-\kappa_1,\ldots,2i_n-\kappa_n+1}
\qquad(1\le s\le d).
\end{aligned}
\]
Combining the componentwise estimates for all $1\le s\le d$ proves the three inclusions in \eqref{eq:estimate}.
Increasing any component constant weakens the coefficient lower bounds and raises the depth threshold in the right lexicographic order, so the conclusions remain valid.
\end{proof}

\begin{proposition}\label{prop:deep}
Choose $\kappa_1,\ldots,\kappa_n$ as in Lemma~\ref{lem:estimate}. For the depths
\[
(i_1,\ldots,i_n)\ge(\kappa_1,\ldots,\kappa_n),
\]
the sets $K_{i_1,\ldots,i_n}=z^{-1}(\mathfrak{p}_{i_1,\ldots,i_n}^{\,d})$ are nested subgroups with intersection $\{e\}$. For every such depth and every integer $m\ge0$,
\begin{equation}\label{eq:deepindex}
K_{i_1,\ldots,i_n}/K_{i_1+1,i_2,\ldots,i_n}\simeq(k^d,+),\qquad
[K_{i_1,\ldots,i_n}:K_{i_1+m,i_2,\ldots,i_n}]=q^{dm}.
\end{equation}
Each $K_{i_1+m,i_2,\ldots,i_n}$ is normal in $K_{i_1,\ldots,i_n}$. For any two depths in this family,
\begin{equation}\label{eq:commutator}
[K_{i_1,\ldots,i_n},K_{j_1,\ldots,j_n}]\subseteq K_{i_1+j_1-\kappa_1,\ldots,i_n+j_n-\kappa_n+1}.
\end{equation}
For such depths with $(j_1,\ldots,j_n)\ge(i_1,\ldots,i_n)$ and $(j_2,\ldots,j_n)>(i_2,\ldots,i_n)$, the index $[K_{i_1,\ldots,i_n}:K_{j_1,\ldots,j_n}]$ is infinite.
\end{proposition}

\begin{proof}
For an allowed depth $(i_1,\ldots,i_n)$ and every $m\ge0$, right lexicographic order and $n\ge2$ give
\[
\begin{aligned}
&(2i_1-\kappa_1,\ldots,2i_n-\kappa_n+1)\\
&\hspace{2em}\ge(i_1,\ldots,i_{n-1},i_n+1)
>(i_1+m,i_2,\ldots,i_n).
\end{aligned}
\]
Consequently,
\[
\mathfrak{p}_{2i_1-\kappa_1,\ldots,2i_n-\kappa_n+1}^{\,d}
\subseteq\mathfrak{p}_{i_1+m,i_2,\ldots,i_n}^{\,d}
\subseteq\mathfrak{p}_{i_1,\ldots,i_n}^{\,d}.
\]
For $g,h\in K_{i_1,\ldots,i_n}$, put $x=z(g)$ and $y=z(h)$. Lemma~\ref{lem:estimate} yields
\[
\begin{aligned}
z(e)&=0\in\mathfrak{p}_{i_1,\ldots,i_n}^{\,d},\\
R(x,y),S(x)&\in\mathfrak{p}_{i_1+m,i_2,\ldots,i_n}^{\,d}\qquad(m\ge0),\\
z(gh)&=x+y+R(x,y)\in\mathfrak{p}_{i_1,\ldots,i_n}^{\,d},\\
z(g^{-1})&=-x+S(x)\in\mathfrak{p}_{i_1,\ldots,i_n}^{\,d}.
\end{aligned}
\]
Thus every $K_{i_1,\ldots,i_n}$ is a subgroup. Nesting and separatedness follow from
\[
\begin{aligned}
(j_1,\ldots,j_n)\ge(i_1,\ldots,i_n)
&\ \Longrightarrow\
\mathfrak{p}_{j_1,\ldots,j_n}^{\,d}\subseteq\mathfrak{p}_{i_1,\ldots,i_n}^{\,d}\\
&\ \Longrightarrow\
K_{j_1,\ldots,j_n}\subseteq K_{i_1,\ldots,i_n},\\
\bigcap_{r\ge0}\mathfrak{p}_{\kappa_1,\ldots,\kappa_{n-1},\kappa_n+r}^{\,d}
&\subseteq\bigcap_{r\ge0}(t_n^{\kappa_n+r}\mathcal{O}_F)^d=\{0\}.
\end{aligned}
\]
Since $z(e)=0$ and $z$ is injective, the intersection of all these subgroups is $\{e\}$.

Fix an allowed depth and $m\ge0$. Increasing the first component preserves the threshold:
\[
(i_1+m,i_2,\ldots,i_n)\ge(i_1,\ldots,i_n)
\ge(\kappa_1,\ldots,\kappa_n).
\]
Define
\[
\begin{aligned}
\theta_m:K_{i_1,\ldots,i_n}
&\longrightarrow
\mathfrak{p}_{i_1,\ldots,i_n}^{\,d}/\mathfrak{p}_{i_1+m,i_2,\ldots,i_n}^{\,d},\\
g&\longmapsto z(g)+\mathfrak{p}_{i_1+m,i_2,\ldots,i_n}^{\,d}.
\end{aligned}
\]
The remainder inclusion above holds for this $m$, so
\[
\begin{aligned}
\theta_m(gh)
&=z(g)+z(h)+R(z(g),z(h))
  +\mathfrak{p}_{i_1+m,i_2,\ldots,i_n}^{\,d}\\
&=\theta_m(g)+\theta_m(h).
\end{aligned}
\]
Every target class has a representative $x\in\mathfrak{p}_{i_1,\ldots,i_n}^{\,d}$, and $g=z^{-1}(x)$ belongs to $K_{i_1,\ldots,i_n}$ with $\theta_m(g)=x+\mathfrak{p}_{i_1+m,i_2,\ldots,i_n}^{\,d}$. Thus $\theta_m$ is surjective. Its kernel is determined by
\[
g\in\ker\theta_m
\ \Longleftrightarrow\ z(g)\in\mathfrak{p}_{i_1+m,i_2,\ldots,i_n}^{\,d}
\ \Longleftrightarrow\ g\in K_{i_1+m,i_2,\ldots,i_n}.
\]
Hence $K_{i_1+m,i_2,\ldots,i_n}$ is normal as the kernel of a group homomorphism. The first isomorphism theorem and the successive reductions in Proposition~\ref{prop:ideal} give
\[
\begin{aligned}
K_{i_1,\ldots,i_n}/K_{i_1+m,i_2,\ldots,i_n}
&\simeq\bigl(\mathfrak{p}_{i_1,\ldots,i_n}/\mathfrak{p}_{i_1+m,i_2,\ldots,i_n}\bigr)^d\\
&\simeq\bigl(\mathcal{O}_{F_1}/\pi_1^m\mathcal{O}_{F_1}\bigr)^d.
\end{aligned}
\]
For $m=1$ the last group is $(k^d,+)$; for every $m\ge0$ it has $(q^m)^d=q^{dm}$ elements. This proves \eqref{eq:deepindex} and all the asserted normality statements. For $m=0$, both quotients are trivial and the index is $1$.

For $g\in K_{i_1,\ldots,i_n}$ and $h\in K_{j_1,\ldots,j_n}$, Lemma~\ref{lem:estimate} gives
\[
z([g,h])=C(z(g),z(h))
\in\mathfrak{p}_{i_1+j_1-\kappa_1,\ldots,i_n+j_n-\kappa_n+1}^{\,d}.
\]
The target depth remains above the threshold, since
\[
\begin{aligned}
&(i_1+j_1-\kappa_1,\ldots,i_n+j_n-\kappa_n+1)\\
&\hspace{2em}\ge(\kappa_1,\ldots,\kappa_{n-1},\kappa_n+1)
>(\kappa_1,\ldots,\kappa_n).
\end{aligned}
\]
Thus the target is a subgroup, and
\[
\langle[g,h]:g\in K_{i_1,\ldots,i_n},\
h\in K_{j_1,\ldots,j_n}\rangle
\subseteq K_{i_1+j_1-\kappa_1,\ldots,i_n+j_n-\kappa_n+1},
\]
which is \eqref{eq:commutator}.

Finally, if $(j_2,\ldots,j_n)>(i_2,\ldots,i_n)$, then for every $m\ge0$,
\[
(j_1,\ldots,j_n)>(i_1+m,i_2,\ldots,i_n),
\qquad
K_{j_1,\ldots,j_n}\subseteq K_{i_1+m,i_2,\ldots,i_n}.
\]
The induced surjection of left-coset sets gives
\[
[K_{i_1,\ldots,i_n}:K_{j_1,\ldots,j_n}]\ge q^{dm}
\quad(m\ge0).
\]
Since $q>1$ and $d>0$, the index is infinite.
\end{proof}

The extension of subgroup volumes rests on a finite partition argument, as in Fesenko~\cite{fesenko-higher} and Waller~\cite{waller}. Let $P$ be a group and let $J\subseteq\mathbb{Z}^n$ be a nonempty set of depths, with the right lexicographic order, satisfying
\[
(i_1,\ldots,i_n)\in J,\quad m\ge0
\quad\Longrightarrow\quad(i_1+m,i_2,\ldots,i_n)\in J.
\]
For $(i_1,\ldots,i_n)\in J$, let $P_{i_1,\ldots,i_n}\le P$ be a subgroup. Assume these subgroups are nested with respect to depth. Let $B$ be a locally compact group with a strictly decreasing basis $(U_{i_1})_{i_1\in\mathbb{Z}}$ of compact open identity subgroups and Haar measure $m_B$. The cosets $hP_{i_1,\ldots,i_n}$, with $h\in P$ and $(i_1,\ldots,i_n)\in J$, are called cylinders. They are pairwise disjoint or nested, and generate the ring $\mathcal{R}(P)$.

\begin{proposition}\label{prop:extension}
Assume that the following two conditions hold.
\begin{enumerate}
\item For every pair of depths
$(i_1,\ldots,i_n),(j_1,\ldots,j_n)\in J$
with $(j_1,\ldots,j_n)\ge(i_1,\ldots,i_n)$,
\[
[P_{i_1,\ldots,i_n}:P_{j_1,\ldots,j_n}]<\infty
\quad\Longleftrightarrow\quad
(j_2,\ldots,j_n)=(i_2,\ldots,i_n).
\]
\item For every pair of depths $(i_1,i_2,\ldots,i_n),(j_1,i_2,\ldots,i_n)\in J$ with $i_1\le j_1$,
\begin{equation}\label{eq:compatibility}
m_B(U_{i_1})=[P_{i_1,\ldots,i_n}:P_{j_1,i_2,\ldots,i_n}]\,m_B(U_{j_1}).
\end{equation}
\end{enumerate}
For formal variables $Y_2,\ldots,Y_n$, the assignment
\begin{equation}\label{eq:Yweight}
\nu(hP_{i_1,\ldots,i_n})=m_B(U_{i_1})\prod_{\ell=2}^nY_\ell^{i_\ell}
\end{equation}
extends uniquely to a finitely additive left-invariant measure on $\mathcal{R}(P)$ with values in $\mathbb{R}((Y_2))\cdots((Y_n))$. Every nonempty ring set has strictly positive measure.
\end{proposition}

\begin{proof}
For $D=hP_{i_1,\ldots,i_n}$, let $w(D)$ be the value in \eqref{eq:Yweight} and set
$\operatorname{lev}(D)=(i_2,\ldots,i_n)$. Equal cylinders have equal underlying subgroups, and
\[
\begin{aligned}
P_{i_1,\ldots,i_n}=P_{j_1,\ldots,j_n}
&\ \Longrightarrow\ (i_2,\ldots,i_n)=(j_2,\ldots,j_n),\\
[P_{i_1,\ldots,i_n}:P_{j_1,\ldots,j_n}]=1
&\ \Longrightarrow\ m_B(U_{i_1})=m_B(U_{j_1})
\ \Longrightarrow\ i_1=j_1.
\end{aligned}
\]
The last implication uses strict decrease of the compact open subgroups and their Haar volumes. Thus $w$ and $\operatorname{lev}$ are well-defined.

For intersecting cylinders, choose $a$ in their intersection:
\[
D=aH,\quad E=aL,\quad
H\subseteq L\ \text{or}\ L\subseteq H
\ \Longrightarrow\
D\subseteq E\ \text{or}\ E\subseteq D.
\]

Consider a finite partition $D=\bigsqcup_{r=1}^ND_r$, with underlying subgroups $H,H_r$. For $a\in D_r$,
\[
D_r=aH_r\subseteq aH=D\ \Longrightarrow\ H_r\subseteq H.
\]
We prove $[H:H_r]<\infty$ by induction on $N$. The case $N=1$ is immediate. For $N>1$, put
\[
L=\max_{\subseteq}\{H_1,\ldots,H_N\}.
\]
Each piece lies in one $L$-coset. A piece with subgroup $L$ is an entire $L$-coset, so $L=H$ would force that piece to equal $D$. Hence
\[
2\le[H:L]\le N.
\]
For each $L$-coset $aL\subseteq D$,
\[
aL=\bigsqcup_{\{r:D_r\subseteq aL\}}D_r,
\qquad
1\le\#\{r:D_r\subseteq aL\}<N.
\]
The induction hypothesis therefore gives
\[
[L:H_r]<\infty,\qquad
[H:H_r]=[H:L][L:H_r]<\infty
\ \Longrightarrow\
\operatorname{lev}(D_r)=\operatorname{lev}(D).
\]

Write
\[
D=hP_{i_1,\ldots,i_n},\qquad
D_r=h_rP_{j_{r,1},i_2,\ldots,i_n},\qquad
a\ge\max\{i_1,j_{1,1},\ldots,j_{N,1}\}.
\]
The closure property of $J$ ensures $(a,i_2,\ldots,i_n)\in J$. Set $M=P_{a,i_2,\ldots,i_n}$,
$w_M=m_B(U_a)\prod_{\ell=2}^nY_\ell^{i_\ell}$. Counting $M$-cosets and using \eqref{eq:compatibility},
\[
\begin{aligned}
[H:M]&=\sum_{r=1}^N[H_r:M],\\
w(D)&=[H:M]w_M
=\sum_{r=1}^N[H_r:M]w_M
=\sum_{r=1}^Nw(D_r).
\end{aligned}
\]

For a shell, prescribe
\begin{equation}\label{eq:shell}
S=D\setminus\bigsqcup_{\ell=1}^mD_\ell,\qquad D_\ell\subsetneq D,
\end{equation}
\begin{equation}\label{eq:shellweight}
\nu(S)=w(D)-\sum_{\ell=1}^m w(D_\ell),
\end{equation}
where the holes are disjoint and $m=0$ is allowed. The partition identity gives
\[
S=\varnothing
\ \Longrightarrow\
D=\bigsqcup_{\ell=1}^mD_\ell
\ \Longrightarrow\
w(D)-\sum_{\ell=1}^m w(D_\ell)=0.
\]

For a finite family of sets, $\operatorname{Max}_{\subseteq}$ denotes the subfamily of inclusion-maximal members. For a finite cylinder family $\mathcal F$, define its immediate descendants and shells by
\[
\begin{aligned}
\mathcal H_{\mathcal F}(C)
&=\operatorname{Max}_{\subseteq}\{E\in\mathcal F:E\subsetneq C\},\\
S_C^{\mathcal F}
&=C\setminus\bigsqcup_{E\in\mathcal H_{\mathcal F}(C)}E,\\
W_{\mathcal F}(C)
&=w(C)-\sum_{E\in\mathcal H_{\mathcal F}(C)}w(E).
\end{aligned}
\]
Nesting gives disjoint decompositions
\[
\bigcup_{C\in\mathcal F}C=\bigsqcup_{C\in\mathcal F}S_C^{\mathcal F},
\qquad
C=\bigsqcup_{\substack{E\in\mathcal F\\E\subseteq C}}S_E^{\mathcal F}.
\]
Thus every set in $\operatorname{ring}\mathcal F$ is a union of these shells.

Let $\mathcal F\subseteq\mathcal G$ be finite cylinder families. For $C\in\mathcal G$, each $D\in\mathcal G$ with $D\subsetneq C$ has a unique immediate parent among the members of $\mathcal G$ contained in $C$: the cylinders strictly containing $D$ and contained in $C$ form a nonempty finite chain, whose smallest member is that parent. Thus each such $w(D)$ occurs exactly once in the double sum below, and
\[
\begin{aligned}
\sum_{\substack{E\in\mathcal G\\E\subseteq C}}W_{\mathcal G}(E)
&=\sum_{\substack{E\in\mathcal G\\E\subseteq C}}w(E)
-\sum_{\substack{E\in\mathcal G\\E\subseteq C}}
 \sum_{D\in\mathcal H_{\mathcal G}(E)}w(D)\\
&=\sum_{\substack{E\in\mathcal G\\E\subseteq C}}w(E)
-\sum_{\substack{D\in\mathcal G\\D\subsetneq C}}w(D)
=w(C).
\end{aligned}
\]
This identity follows directly from the weight definitions.

For $C\in\mathcal F$, put
\[
\mathcal G_C=
\{E\in\mathcal G:E\subseteq C,\ 
 E\nsubseteq D\text{ for every }D\in\mathcal H_{\mathcal F}(C)\}.
\]
The $\mathcal G$-shells partition $C$. Each old hole $D\in\mathcal H_{\mathcal F}(C)$ is the union of exactly those shells indexed by $E\subseteq D$. Removing these disjoint holes gives
\[
S_C^{\mathcal F}=\bigsqcup_{E\in\mathcal G_C}S_E^{\mathcal G}.
\]
Applying the preceding weight identity to $C$ and to each old hole gives
\[
\begin{aligned}
W_{\mathcal F}(C)
&=w(C)-\sum_{D\in\mathcal H_{\mathcal F}(C)}w(D)\\
&=\sum_{\substack{E\in\mathcal G\\E\subseteq C}}W_{\mathcal G}(E)
-\sum_{D\in\mathcal H_{\mathcal F}(C)}
 \sum_{\substack{E\in\mathcal G\\E\subseteq D}}W_{\mathcal G}(E)\\
&=\sum_{E\in\mathcal G_C}W_{\mathcal G}(E).
\end{aligned}
\]
An empty shell has weight zero by the finite partition identity already proved. Empty shells may therefore be omitted from these decompositions and sums.

For $S\in\operatorname{ring}\mathcal F$, define
\[
\nu_{\mathcal F}(S)
=\sum_{\substack{C\in\mathcal F\\
\varnothing\ne S_C^{\mathcal F}\subseteq S}}W_{\mathcal F}(C).
\]
Each nonempty $\mathcal G$-shell contained in $S$ occurs in exactly one of the disjoint refinements above. Summing their weights consequently gives
\[
\nu_{\mathcal F}(S)
=\sum_{\substack{E\in\mathcal G\\
\varnothing\ne S_E^{\mathcal G}\subseteq S}}W_{\mathcal G}(E)
=\nu_{\mathcal G}(S).
\]
Two representations therefore agree on their common refinement:
\[
\nu_{\mathcal F_1}(S)
=\nu_{\mathcal F_1\cup\mathcal F_2}(S)
=\nu_{\mathcal F_2}(S).
\]
This defines $\nu$ on $\mathcal{R}(P)$. On a common shell decomposition,
\[
S\cap T=\varnothing
\ \Longrightarrow\
\nu(S\bigsqcup T)=\nu(S)+\nu(T).
\]
Moreover, $w(aD)=w(D)$ implies $\nu(aS)=\nu(S)$. Formula \eqref{eq:shellweight} forces every extension to have these values, proving uniqueness.

For positivity, take a nonempty shell $S$ with parent $D=hP_{i_1,\ldots,i_n}$. Let $L_\ell$ underlie $D_\ell$, and set
\[
T=\{\ell:\operatorname{lev}(D_\ell)=(i_2,\ldots,i_n)\}.
\]
Choose $a$ at least as large as the first indices of the parent and these holes; set $M=P_{a,i_2,\ldots,i_n}$. The number of parent $M$-cosets left by these holes is
\[
b=[P_{i_1,\ldots,i_n}:M]-\sum_{\ell\in T}[L_\ell:M]\ge1,
\]
since $S\ne\varnothing$. Thus
\[
\nu(S)=
b\,m_B(U_a)\prod_{r=2}^nY_r^{i_r}
-\sum_{\ell\notin T}w(D_\ell).
\]
Every term in the last sum has level strictly greater than $(i_2,\ldots,i_n)$. The leading real coefficient is therefore
\[
\operatorname{lc}(\nu(S))=b\,m_B(U_a)>0
\ \Longrightarrow\ \nu(S)>0.
\]
Finally, $S=\bigsqcup_{r=1}^NS_r$ with nonempty shells gives
$\nu(S)=\sum_{r=1}^N\nu(S_r)>0$.
\end{proof}

\begin{corollary}\label{cor:analyticfamily}
Let $A$ be an $F$-analytic group of dimension $d>0$. Choose $\kappa_1,\ldots,\kappa_n$ as in Proposition~\ref{prop:deep} and set
\begin{equation}\label{eq:Hdef}
\begin{aligned}
H^z_{i_1,\ldots,i_n}
&:=K_{i_1+\kappa_1,\ldots,i_n+\kappa_n}\\
&=z^{-1}(\mathfrak{p}_{i_1+\kappa_1,\ldots,i_n+\kappa_n}^{\,d}),
\qquad (i_1,\ldots,i_n)\in\mathcal I.
\end{aligned}
\end{equation}
The family $P_{i_1,\ldots,i_n}=H^z_{i_1,\ldots,i_n}$ satisfies the hypotheses of Proposition~\ref{prop:extension} with $P=A$, $J=\mathcal I$, $B=(F_1^d,+)$, and $U_{i_1}=\pi_1^{i_1}\mathcal{O}_{F_1}^d$, where $m_B(\mathcal{O}_{F_1}^d)=1$. More precisely, for every $(i_1,\ldots,i_n)\in\mathcal I$ and $m\ge0$,
\begin{equation}\label{eq:verifiedindex}
[H^z_{i_1,\ldots,i_n}:H^z_{i_1+m,i_2,\ldots,i_n}]=q^{dm},\qquad
m_B(U_{i_1})=q^{-di_1}.
\end{equation}
\end{corollary}

\begin{proof}
The shift in \eqref{eq:Hdef} preserves depth comparisons and sends $\mathcal I$ into the range $(i_1+\kappa_1,\ldots,i_n+\kappa_n)\ge(\kappa_1,\ldots,\kappa_n)$. Also, $\mathcal I$ is closed under increasing $i_1$. Proposition~\ref{prop:deep} therefore gives nested subgroups and, for depths in $\mathcal I$ with $(j_1,\ldots,j_n)\ge(i_1,\ldots,i_n)$,
\[
[H^z_{i_1,\ldots,i_n}:H^z_{j_1,\ldots,j_n}]
=
\begin{cases}
q^{d(j_1-i_1)},&(j_2,\ldots,j_n)=(i_2,\ldots,i_n),\\
\infty,&(j_2,\ldots,j_n)>(i_2,\ldots,i_n).
\end{cases}
\]
In $B=(F_1^d,+)$, the subgroups $U_{i_1}=\pi_1^{i_1}\mathcal{O}_{F_1}^d$ form a strictly decreasing compact open basis. With $m_B(U_0)=1$, Haar measure gives $m_B(U_{i_1})=q^{-di_1}$ for every $i_1\in\mathbb{Z}$ \cite[Chapter~I]{weil}. Thus, for $j_1\ge i_1$,
\[
\begin{aligned}
[H^z_{i_1,\ldots,i_n}:H^z_{j_1,i_2,\ldots,i_n}]\,m_B(U_{j_1})
&=q^{d(j_1-i_1)}q^{-dj_1}\\
&=q^{-di_1}=m_B(U_{i_1}),
\end{aligned}
\]
which verifies \eqref{eq:compatibility}. The index classification verifies the other hypothesis of Proposition~\ref{prop:extension}, and taking $j_1=i_1+m$ gives \eqref{eq:verifiedindex}.
\end{proof}

\begin{theorem}\label{thm:main}
Let $G/F$ be a smooth algebraic group of dimension $d$. If $d>0$, choose an analytic chart $z:U\xrightarrow{\sim}V\subset F^d$ on an open neighborhood of the identity $e$, with $z(e)=0$. There are integers $\kappa_1,\ldots,\kappa_{n-1}\ge0$ and $\kappa_n\ge1$ such that all balls $\mathfrak{p}_{i_1+\kappa_1,\ldots,i_n+\kappa_n}^{\,d}$ with $(i_1,\ldots,i_n)\in\mathcal I$ lie in $V$, and
\[
H^z_{i_1,\ldots,i_n}:=z^{-1}(\mathfrak{p}_{i_1+\kappa_1,\ldots,i_n+\kappa_n}^{\,d}),
\qquad (i_1,\ldots,i_n)\in\mathcal I,
\]
are subgroups. Their left cosets are called \emph{congruence cylinders}. The ring
\begin{equation}\label{eq:mainring}
\mathcal{R}_z(G(F))=\operatorname{ring}\{gH^z_{i_1,\ldots,i_n}:g\in G(F),\ (i_1,\ldots,i_n)\in\mathcal I\}
\end{equation}
has a unique finitely additive left-invariant measure $\mu^0_{G,z}:\mathcal{R}_z(G(F))\to\mathbb R_n$ with
\begin{equation}\label{eq:mainweight}
\boxed{\ \mu^0_{G,z}(gH^z_{i_1,\ldots,i_n})=q^{-di_1}\prod_{\ell=2}^nX_\ell^{di_\ell}.\ }
\end{equation}
Every nonempty measurable set has strictly positive measure. For $d=0$, counting measure on finite subsets has these properties.
\end{theorem}

\begin{proof}[Proof of Theorem~\ref{thm:main}]
Let $A$ be any $F$-analytic group of dimension $d>0$. Corollary~\ref{cor:analyticfamily} and Proposition~\ref{prop:extension} give a finitely additive left-invariant measure $\nu$ on
\[
\mathcal{R}_z(A)=\operatorname{ring}\{aH^z_{i_1,\ldots,i_n}:a\in A,\ (i_1,\ldots,i_n)\in\mathcal I\}
\]
with cylinder values
\[
\nu(aH^z_{i_1,\ldots,i_n})=q^{-di_1}\prod_{\ell=2}^nY_\ell^{i_\ell},
\qquad S\ne\varnothing\ \Longrightarrow\ \nu(S)>0.
\]
The field embedding
\begin{equation}\label{eq:substitution}
\varepsilon_d:\mathbb{R}((Y_2))\cdots((Y_n))\hookrightarrow\mathbb R_n,
\qquad Y_\ell\longmapsto X_\ell^d\quad(2\le\ell\le n),
\end{equation}
preserves the order of exponent tuples and the leading real coefficient, since $d>0$. Hence $\mu^0_{A,z}:=\varepsilon_d\circ\nu$ is finitely additive, left-invariant, and strictly positive on nonempty ring sets, with cylinder values
\[
\mu^0_{A,z}(aH^z_{i_1,\ldots,i_n})
=q^{-di_1}\prod_{\ell=2}^nX_\ell^{di_\ell}.
\]
Every ring set is a finite disjoint union of shells, and finite additivity forces
\[
\mu\left(D\setminus\bigsqcup_{\ell=1}^mD_\ell\right)
=\mu(D)-\sum_{\ell=1}^m\mu(D_\ell).
\]
Thus the prescribed cylinder values also give uniqueness.

For $d>0$, smoothness makes $G(F)$ a $d$-dimensional $F$-analytic group, so taking $A=G(F)$ proves the assertion. For $d=0$, the identity chart is $\{e\}$; its translates generate the ring of finite subsets, on which $\mu(S)=\#S$ has the required properties.
\end{proof}

For $(i_1,\ldots,i_n)\ge(\kappa_1,\ldots,\kappa_n)$, the absolute coordinate normalization is
\begin{equation}\label{eq:absnormalization}
\begin{aligned}
\mu^{\mathrm{abs}}_{A,z}(aK_{i_1,\ldots,i_n})
&=q^{-di_1}\prod_{\ell=2}^nX_\ell^{di_\ell},\\
\mu^{\mathrm{abs}}_{A,z}
&=q^{-d\kappa_1}
\left(\prod_{\ell=2}^nX_\ell^{d\kappa_\ell}\right)\mu^0_{A,z}.
\end{aligned}
\end{equation}
Indeed, the shifted reference cylinder has absolute volume
\[
\mu^{\mathrm{abs}}_{A,z}(H^z_{0,\ldots,0})
=\mu^{\mathrm{abs}}_{A,z}(K_{\kappa_1,\ldots,\kappa_n})
=q^{-d\kappa_1}\prod_{\ell=2}^nX_\ell^{d\kappa_\ell},
\]
whereas $\mu^0_{A,z}(H^z_{0,\ldots,0})=1$.
Giving volume $1$ to a reference cylinder of depth $(j_1,\ldots,j_n)$ instead gives $q^{-d(i_1-j_1)}\prod_{\ell=2}^nX_\ell^{d(i_\ell-j_\ell)}$.

Fix a normalization and call the measure $\mu$. For a set $C$, let $\mathbf{1}_C$ be its characteristic function. Define the cylinder-simple functions and their integrals by
\begin{equation}\label{eq:simpleintegral}
\mathcal{S}_z(A;\Lambda_n)=\operatorname{span}_{\Lambda_n}\{\mathbf{1}_C:C\in\mathcal{R}_z(A)\},
\qquad
\int_A\sum_{\ell=1}^M c_\ell\mathbf{1}_{C_\ell}\,d\mu
=\sum_{\ell=1}^M c_\ell\mu(C_\ell).
\end{equation}
For two presentations $f=\sum_r c_r\mathbf{1}_{C_r}=\sum_s d_s\mathbf{1}_{D_s}$, take a common partition into nonempty atoms $E_t$. Equality on each atom gives
\[
\begin{aligned}
\sum_r c_r\mu(C_r)
&=\sum_t\left(\sum_{r:E_t\subseteq C_r}c_r\right)\mu(E_t)\\
&=\sum_t\left(\sum_{s:E_t\subseteq D_s}d_s\right)\mu(E_t)
=\sum_s d_s\mu(D_s).
\end{aligned}
\]
Thus the integral is well-defined. Left invariance gives
\[
\int_A f(a^{-1}x)\,d\mu(x)=\int_A f(x)\,d\mu(x).
\]

\section{Properties of the measure}\label{sec:properties}

Fix an analytic group $A$ of dimension $d>0$ and a family $H^z_{i_1,\ldots,i_n}$ as in \eqref{eq:Hdef}. We write $f'(a)$ for the differential of an analytic map $f$ at $a$.

\begin{proposition}\label{prop:normalization}
Every finitely additive left-invariant function $\nu$ on $\mathcal{R}_z(A)$ with values in a characteristic-zero field satisfies, for $(i_1,\ldots,i_n)\in\mathcal I$,
\begin{equation}\label{eq:cs}
\begin{aligned}
\nu(aH^z_{i_1,\ldots,i_n})&=q^{-di_1}c_{i_2,\ldots,i_n},\\
c_{i_2,\ldots,i_n}&=\nu(H^z_{0,i_2,\ldots,i_n}).
\end{aligned}
\end{equation}
Conversely, every choice of the $c_{i_2,\ldots,i_n}$, indexed by $(i_2,\ldots,i_n)\ge(0,\ldots,0)$, gives a unique such function.
\end{proposition}

\begin{proof}
Fix an allowed depth $(i_1,\ldots,i_n)\in\mathcal I$ and put $r=\max\{i_1,0\}$. Corollary~\ref{cor:analyticfamily}, finite additivity, and left invariance give
\[
\begin{aligned}
\nu(H^z_{i_1,\ldots,i_n})
&=q^{d(r-i_1)}\nu(H^z_{r,i_2,\ldots,i_n}),\\
c_{i_2,\ldots,i_n}
&=q^{dr}\nu(H^z_{r,i_2,\ldots,i_n}).
\end{aligned}
\]
Division by $q^{dr}$ is valid in characteristic zero; hence
$\nu(aH^z_{i_1,\ldots,i_n})=q^{-di_1}c_{i_2,\ldots,i_n}$.

Conversely, prescribe these cylinder values for arbitrary $c_{i_2,\ldots,i_n}$. The proof of Proposition~\ref{prop:extension} shows that every finite partition of a cylinder into cylinders has a single higher level. Refinement to a common first index $a\ge j_1$ respects the assigned weights because
\[
q^{-dj_1}c_{i_2,\ldots,i_n}
=q^{d(a-j_1)}q^{-da}c_{i_2,\ldots,i_n}.
\]
The shell and common-refinement construction in that proof therefore gives a unique finitely additive left-invariant extension. This part of the construction does not require positivity.
\end{proof}

Theorem~\ref{thm:main} chooses $c_{i_2,\ldots,i_n}=\prod_{\ell=2}^nX_\ell^{di_\ell}$. Same-level refinement preserves the integral because
\[
q^{-di_1}\prod_{\ell=2}^nX_\ell^{di_\ell}
=q^{dm}q^{-d(i_1+m)}\prod_{\ell=2}^nX_\ell^{di_\ell}.
\]
Restriction to a family with a larger depth threshold preserves the absolute normalization.

A \emph{rank-$n$ lattice} in $F^d$ is a module $L\mathcal{O}^d$ with $L\in\operatorname{GL}_d(F)$; an affine lattice coset is a translate of such a module. Put
\begin{equation}\label{eq:lattice}
\operatorname{vol}(L\mathfrak{p}_{i_1,\ldots,i_n}^{\,d})=\left|\det L\right|_n\,q^{-di_1}\prod_{\ell=2}^nX_\ell^{di_\ell}.
\end{equation}
For an affine lattice coset $a+M$, set $\operatorname{vol}(a+M)=\operatorname{vol}(M)$.
Absorb $t_1^{i_1}\cdots t_n^{i_n}$ into the lattice basis. Independence of the basis follows from
\[
\begin{aligned}
L_1\mathcal{O}^d=L_2\mathcal{O}^d
&\ \Longrightarrow\ L_2^{-1}L_1\in\operatorname{GL}_d(\mathcal{O})\\
&\ \Longrightarrow\ v(\det(L_2^{-1}L_1))=0\\
&\ \Longrightarrow\ \left|\det L_1\right|_n=\left|\det L_2\right|_n.
\end{aligned}
\]

For a lattice inclusion, elimination over the valuation ring gives coordinates with
\[
M=\mathcal{O}^d,\qquad M'=\bigoplus_{s=1}^d a_s\mathcal{O},\qquad v(a_s)\ge0.
\]
Proposition~\ref{prop:ideal} yields
\[
[M:M']<\infty
\quad\Longleftrightarrow\quad
(v_2(a_s),\ldots,v_n(a_s))=(0,\ldots,0)\quad(1\le s\le d).
\]
When these equivalent conditions hold,
\[
[M:M']=q^{\sum_s v_1(a_s)},\qquad
\operatorname{vol}(M')=[M:M']^{-1}\operatorname{vol}(M).
\]
In the infinite-index case,
\[
\sum_{s=1}^d(v_2(a_s),\ldots,v_n(a_s))>(0,\ldots,0),
\]
so $\operatorname{vol}(M')/\operatorname{vol}(M)$ has a strictly positive higher exponent tuple.

Nonempty affine intersections reduce to lattice intersections:
\[
x\in(a+M)\cap(b+M')
\ \Longrightarrow\
(a+M)\cap(b+M')=x+(M\cap M').
\]
Diagonalizing the pair gives
\[
\mathcal{O}^d\cap\bigoplus_s a_s\mathcal{O}
=\bigoplus_s(\mathcal{O}\cap a_s\mathcal{O}),
\qquad
\mathcal{O}\cap a_s\mathcal{O}\in\{\mathcal{O},a_s\mathcal{O}\}.
\]
Hence every nonempty finite intersection is again an affine lattice coset.

For $S=D\setminus\bigcup_{\ell=1}^mD_\ell$, with affine lattice holes $D_\ell\subseteq D$, prescribe the inclusion--exclusion weight
\[
w(S)=
\sum_{T\subseteq\{1,\ldots,m\}}(-1)^{|T|}
\operatorname{vol}\left(D\cap\bigcap_{\ell\in T}D_\ell\right),
\qquad \operatorname{vol}(\varnothing)=0.
\]
For $T=\varnothing$, the intersection denotes $D$.

The finite-index holes admit a common finite coset refinement. Write the cells surviving those holes as $C_1,\ldots,C_b$. In each surviving cell, every remaining hole has infinite index. Identify the cell with $\mathcal{O}^d$ and reduce coordinatewise by $\rho_1$.

For each such hole $E_r$, diagonalization supplies a nonzero affine linear polynomial $\ell_r$ over $F_1$ with
\[
\rho_1(E_r)\subseteq\{u\in F_1^d:\ell_r(u)=0\}.
\]
The ring $\mathcal{O}_{F_1}$ is infinite, so
\[
0\ne\prod_r\ell_r
\ \Longrightarrow\
\exists u\in\mathcal{O}_{F_1}^d:\prod_r\ell_r(u)\ne0
\ \Longrightarrow\
\mathcal{O}^d\setminus\bigcup_rE_r\ne\varnothing.
\]
All nonempty intersections involving one of these holes have a larger exponent tuple than that of the cell. Thus
\[
\operatorname{lc}\left(w\left(C_t\setminus\bigcup_\ell D_\ell\right)\right)
=\operatorname{lc}(\operatorname{vol}(C_t))>0.
\]
Consequently,
\[
S=\varnothing\ \Longleftrightarrow\ b=0
\ \Longrightarrow\ w(S)=0,
\qquad
S\ne\varnothing\ \Longrightarrow\ w(S)>0.
\]
Common finite Boolean refinements now give a well-defined, positive, finitely additive extension on the ring generated by affine lattice cosets. We continue to denote this extension by $\operatorname{vol}$.

\begin{lemma}\label{lem:image}
Let $f:W\to F^d$ be analytic on an open neighborhood of $0$, with $f(0)=0$ and $L=f'(0)\in\operatorname{GL}_d(F)$. There are integers $\kappa_1,\ldots,\kappa_{n-1}\ge0$ and $\kappa_n\ge1$ such that, for every depth
\[
(i_1,\ldots,i_n)\ge(\kappa_1,\ldots,\kappa_n),
\]
the ball $\mathfrak{p}_{i_1,\ldots,i_n}^{\,d}$ lies in $W$ and
\begin{equation}\label{eq:image}
f(\mathfrak{p}_{i_1,\ldots,i_n}^{\,d})=L\mathfrak{p}_{i_1,\ldots,i_n}^{\,d}.
\end{equation}
The map $f$ is a bijection between these sets. There is a neighborhood $V\subseteq W$ of $0$ such that, for every $a\in V$ and every depth $(j_1,\ldots,j_n)\ge(\kappa_1,\ldots,\kappa_n)$, the affine ball lies in $W$ and
\[
f(a+\mathfrak{p}_{j_1,\ldots,j_n}^{\,d})=f(a)+f'(a)\mathfrak{p}_{j_1,\ldots,j_n}^{\,d},
\qquad \left|\det f'(a)\right|_n=\left|\det L\right|_n.
\]
\end{lemma}

\begin{proof}
Since $\det f'(0)=\det L\ne0$, choose a neighborhood $V_0$ of $0$ and a common outer polydisc on which $a+h\in W$ and $f'(a)$ is invertible. The maps
\[
\begin{aligned}
g_a(h)&=f'(a)^{-1}\bigl(f(a+h)-f(a)\bigr)=h+E_a(h),\\
g_a(0)&=0,\qquad g_a'(0)=\operatorname{id}_{F^d},\\
E_a(0)&=0,\qquad E_a'(0)=0
\end{aligned}
\]
are jointly analytic in $(a,h)$ for $a\in V_0$ and small $h$. Each component of $E_a(h)$ has only terms of degree at least two in $h$.

The identities
\[
u^r-v^r=(u-v)\sum_{t=0}^{r-1}u^{r-1-t}v^t
\qquad(r\ge1)
\]
and coordinatewise telescoping give, on a smaller common convergence polydisc,
\[
E_{a,s}(h)-E_{a,s}(k)
=\sum_{r=1}^d(h_r-k_r)b_{sr}(a,h,k),
\qquad b_{sr}(a,0,0)=0.
\]
The factor series are jointly analytic; after one factor $h_r-k_r$ is removed, their terms still have positive total degree in $(h,k)$. By joint continuity, shrink $V_0$ to a neighborhood $V$ of $0$ and choose integers $\kappa_1,\ldots,\kappa_{n-1}\ge0$ and $\kappa_n\ge1$ such that, for every $a\in V$ and $h,k\in\mathfrak{p}_{\kappa_1,\ldots,\kappa_n}^{\,d}$,
\[
\begin{aligned}
a+\mathfrak{p}_{\kappa_1,\ldots,\kappa_n}^{\,d}&\subseteq W,\\
b_{sr}(a,h,k)&\in t_n^2\mathcal{O}_F\qquad(1\le s,r\le d).
\end{aligned}
\]
The constants are independent of $a$. Since
\[
v(b_{sr}(a,h,k))>(0,\ldots,0,1),
\]
the valuation inequality for finite sums gives
\[
v(E_a(h)-E_a(k))
\ge v(h-k)+(0,\ldots,0,1)\qquad(h\ne k).
\]
Use the real-valued outer absolute value and maximum coordinate norm
\[
|u|_{\mathrm{out}}=2^{-v_n(u)}\quad(u\ne0),\qquad
|0|_{\mathrm{out}}=0,\qquad
\|h\|_{\mathrm{out}}=\max_{1\le s\le d}|h_s|_{\mathrm{out}}.
\]
The same coefficient bound gives, with $c=|t_n|_{\mathrm{out}}^2<1$,
\[
\|E_a(h)-E_a(k)\|_{\mathrm{out}}
\le c\|h-k\|_{\mathrm{out}}.
\]

Fix $a\in V$ and $(i_1,\ldots,i_n)\ge(\kappa_1,\ldots,\kappa_n)$. Nesting gives
\[
\mathfrak{p}_{i_1,\ldots,i_n}^{\,d}
\subseteq\mathfrak{p}_{\kappa_1,\ldots,\kappa_n}^{\,d},
\]
so all preceding estimates hold on this ball. For $b\in\mathfrak{p}_{i_1,\ldots,i_n}^{\,d}$, put $T_{a,b}(h)=b-E_a(h)$. Taking $k=0$ and using $E_a(0)=0$ gives
\[
\begin{aligned}
h\in\mathfrak{p}_{i_1,\ldots,i_n}^{\,d}
&\ \Longrightarrow\
v(E_a(h))\ge v(h)+(0,\ldots,0,1)\\
&\ \Longrightarrow\
E_a(h),\ g_a(h),\ T_{a,b}(h)
\in\mathfrak{p}_{i_1,\ldots,i_n}^{\,d}.
\end{aligned}
\]
The case $h=0$ follows directly from $E_a(0)=0$.

Starting with $h^{(0)}=0$, set $h^{(r+1)}=T_{a,b}(h^{(r)})$ for $r\ge0$. The self-map property keeps all iterates in the ball, and the contraction estimate implies
\[
\begin{aligned}
\|h^{(r+1)}-h^{(r)}\|_{\mathrm{out}}
&\le c^r\|h^{(1)}-h^{(0)}\|_{\mathrm{out}},\\
\|h^{(s)}-h^{(r)}\|_{\mathrm{out}}
&\le\max_{r\le t<s}\|h^{(t+1)}-h^{(t)}\|_{\mathrm{out}}\\
&\le c^r\|h^{(1)}-h^{(0)}\|_{\mathrm{out}}
\qquad(s>r).
\end{aligned}
\]
Thus the iterates are Cauchy. Since $F^d$ is outer-complete and the ball is outer-closed, they converge to some $h$ in the same ball. Continuity of $T_{a,b}$ gives
\[
h=T_{a,b}(h)
\ \Longleftrightarrow\ h+E_a(h)=b
\ \Longleftrightarrow\ g_a(h)=b.
\]
For two solutions $h,h'$,
\[
\|h-h'\|_{\mathrm{out}}
\le c\|h-h'\|_{\mathrm{out}},\qquad c<1
\ \Longrightarrow\ h=h'.
\]
This proves bijectivity of $g_a$ on every ball at the stated depths. Undoing its normalization gives, bijectively,
\[
\begin{aligned}
g_a(\mathfrak{p}_{i_1,\ldots,i_n}^{\,d})
&=\mathfrak{p}_{i_1,\ldots,i_n}^{\,d},\\
f(a+\mathfrak{p}_{i_1,\ldots,i_n}^{\,d})
&=f(a)+f'(a)\mathfrak{p}_{i_1,\ldots,i_n}^{\,d}.
\end{aligned}
\]
Taking $a=0$, with $f(0)=0$ and $f'(0)=L$, proves \eqref{eq:image} as well.

Finally, shrink $V$ so that
\[
\frac{\det f'(a)}{\det L}\in1+\mathfrak m_F
\ \Longrightarrow\
v\left(\frac{\det f'(a)}{\det L}\right)=(0,\ldots,0)
\ \Longrightarrow\
\left|\det f'(a)\right|_n=\left|\det L\right|_n.
\]
The middle implication follows by successive reduction of an element with outer residue $1$.
\end{proof}

For identity charts $z,z'$, put $f=z'\circ z^{-1}$ and $L=f'(0)$. Let $\lambda_z,\lambda_{z'}$ be copies of the lattice measure $\operatorname{vol}$ on the respective coordinate domains; thus $\lambda_z(C)=\operatorname{vol}(C)$ for a measurable coordinate set $C$.

On a sufficiently small overlap, Lemma~\ref{lem:image} gives, for $C=a+\mathfrak{p}_{j_1,\ldots,j_n}^{\,d}$,
\[
\begin{aligned}
f(C)&=f(a)+f'(a)\mathfrak{p}_{j_1,\ldots,j_n}^{\,d},\\
\lambda_{z'}(f(C))
&=\left|\det f'(a)\right|_n\,\lambda_z(C)
=\left|\det L\right|_n\,\lambda_z(C).
\end{aligned}
\]
Let $\mathcal D_z$ be the ring generated by these scalar balls, and write $f(\mathcal D_z)=\{f(S):S\in\mathcal D_z\}$. Finite additivity gives the pushforward identity
\[
(f_*\lambda_z)(S):=\lambda_z(f^{-1}(S))
=\left|\det L\right|_n^{-1}\lambda_{z'}(S)
\qquad(S\in f(\mathcal D_z)).
\]

For a nonzero left-invariant top-degree form $\omega$ on $A$, write its expression in the chart $z$ as
\[
\omega=h_z(x)\,dx_1\wedge\cdots\wedge dx_d.
\]
After shrinking the chart,
\[
\frac{h_z(x)}{h_z(0)}\in1+\mathfrak m_F
\ \Longrightarrow\
\left|h_z(x)\right|_n=\left|h_z(0)\right|_n.
\]
Transport the weighted coordinate measure to the group by $z^{-1}$. In these coordinates, define $\lambda_\omega$ by
\begin{equation}\label{eq:gauge}
d\lambda_\omega=\left|h_z(0)\right|_n\,d\lambda_z.
\end{equation}
The coordinate transformation law gives
\[
\begin{aligned}
h_z(0)&=h_{z'}(0)\det L,\\
\left|h_z(0)\right|_n\,f_*\lambda_z
&=\left|h_{z'}(0)\right|_n\left|\det L\right|_n
  \left|\det L\right|_n^{-1}\lambda_{z'}\\
&=\left|h_{z'}(0)\right|_n\,\lambda_{z'}.
\end{aligned}
\]
Thus \eqref{eq:gauge} agrees on the corresponding rings $\mathcal D_z$ and $f(\mathcal D_z)$ after transport to the group.

To identify this local measure with the cylinder measure, put $\psi_a(x)=z(az^{-1}(x))$ for $a$ near $e$. Left invariance of $\omega$ gives
\[
h_z(z(a))\det\psi_a'(0)=h_z(0),
\qquad \left|\det\psi_a'(0)\right|_n=1.
\]
The contraction argument of Lemma~\ref{lem:image} applies uniformly to the jointly analytic family $\psi_a(x)-\psi_a(0)$ after shrinking the identity neighborhood and increasing the component threshold. For all corresponding depths $(i_1,\ldots,i_n)\ge(\kappa_1,\ldots,\kappa_n)$,
\[
\begin{aligned}
z(aK_{i_1,\ldots,i_n})
&=z(a)+\psi_a'(0)\mathfrak{p}_{i_1,\ldots,i_n}^{\,d},\\
\lambda_z\bigl(z(aK_{i_1,\ldots,i_n})\bigr)
&=\left|\det\psi_a'(0)\right|_n\operatorname{vol}(\mathfrak{p}_{i_1,\ldots,i_n}^{\,d})
=\operatorname{vol}(\mathfrak{p}_{i_1,\ldots,i_n}^{\,d}).
\end{aligned}
\]
Finite additivity identifies the weighted coordinate measure with $\left|h_z(0)\right|_n\mu^{\mathrm{abs}}_{A,z}$ on this local cylinder ring. Left translation extends this identification to all congruence cylinders at these depths.

For a parameter change at stage $2\le\ell\le n$, let
\[
\pi_\ell'=a\pi_\ell,\quad a\in\mathcal{O}_{F_\ell}^{\times},
\qquad
v(\bar a)=(c_1,\ldots,c_{\ell-1}).
\]
For an element with original valuation $(i_1,\ldots,i_n)$, let $(i_1',\ldots,i_n')$ be its valuation for the new parameters, and let $X_2',\ldots,X_n'$ be the corresponding formal variables. Keeping lower parameters fixed gives
\[
i_s'=
\begin{cases}
i_s-c_si_\ell,&s<\ell,\\
i_s,&s\ge\ell.
\end{cases}
\]
The order-preserving coefficient isomorphism is
\begin{equation}\label{eq:parameter}
X_\ell'\longmapsto q^{-c_1}\Bigl(\prod_{s=2}^{\ell-1}X_s^{c_s}\Bigr)X_\ell,
\qquad X_s'\longmapsto X_s\quad(s\ne\ell).
\end{equation}
Its effect on the varying factors is
\[
q^{-i_1'}\prod_{s=2}^{\ell-1}X_s^{i_s'}
\left(q^{-c_1}\prod_{s=2}^{\ell-1}X_s^{c_s}X_\ell\right)^{i_\ell}
=q^{-i_1}\prod_{s=2}^{\ell}X_s^{i_s}.
\]
The remaining factors are unchanged. Iteration identifies all moduli and volumes; changing lifts of a fixed residue parameter leaves the valuation unchanged.

For a smooth connected affine group $G/F$, put
\[
\operatorname{Lie} G=T_eG,\qquad
\delta_G(g)=\det(\operatorname{Ad}(g);\operatorname{Lie} G),\qquad R_g(x)=xg.
\]
Here $\operatorname{Ad}(g)$ is the derivative at $e$ of $x\mapsto gxg^{-1}$, and $R_g^*$ denotes pullback of differential forms by $R_g$. Fix $g\in G(F)$ and put
\[
c_g(h)=g^{-1}hg,\qquad f_g=z\circ c_g\circ z^{-1},\qquad L_g=f_g'(0).
\]
Increase the component constants $\kappa_1,\ldots,\kappa_n$, depending on $g$, so that Lemma~\ref{lem:image} applies to both $f_g$ and $f_{g^{-1}}$. For every depth $(i_1,\ldots,i_n)\ge(\kappa_1,\ldots,\kappa_n)$,
\[
z(g^{-1}K_{i_1,\ldots,i_n}g)=L_g\mathfrak{p}_{i_1,\ldots,i_n}^{\,d},
\qquad \det L_g=\delta_G(g)^{-1}.
\]
The conjugated family is the scalar-ball family for the chart $z\circ c_g^{-1}$. Right translation is understood between the ring generated by the original cylinders and the ring generated by their right translates, each equipped with its corresponding gauge measure. The lattice volume formula and left invariance give
\[
\begin{aligned}
K_{i_1,\ldots,i_n}g&=g(g^{-1}K_{i_1,\ldots,i_n}g),\\
R_g^*\omega&=\delta_G(g)^{-1}\omega,\\
\lambda_\omega(Sg)&=\left|\delta_G(g)\right|_n^{-1}\lambda_\omega(S).
\end{aligned}
\]
The last identity extends from cylinders to their ring by finite additivity. Applying it with $g^{-1}$ to a cylinder indicator gives
\[
\int\mathbf{1}_S(xg)\,d\lambda_\omega(x)
=\lambda_\omega(Sg^{-1})
=\left|\delta_G(g)\right|_n\,\lambda_\omega(S).
\]
For every finite linear combination $f$ of indicators of cylinders above this threshold, finite linearity gives
\begin{equation}\label{eq:modularintegral}
\int f(xg)\,d\lambda_\omega(x)
=\left|\delta_G(g)\right|_n\int f(x)\,d\lambda_\omega(x).
\end{equation}
For connected reductive $G$, over a splitting field the nonzero adjoint weights on a maximal torus $T$ occur in opposite pairs \cite[Chapter~21]{milne}:
\[
\delta_G|_T=\prod_{\alpha\in R^+}\alpha\,\alpha^{-1}=1
\ \Longrightarrow\ \delta_G=1
\ \Longrightarrow\ \lambda_\omega(Sg)=\lambda_\omega(S).
\]
Here $R^+$ is a positive root system; characters of a connected reductive group are determined by their restriction to $T$.

For product integration over valuation fields, see \cite{morrow-product,mori}. Let $A_1,A_2$ be $F$-analytic groups of dimensions $d_1,d_2$, with cylinder rings $\mathcal{R}_1,\mathcal{R}_2$ and measures $\mu_1,\mu_2$. On the ring generated by rectangles, set
\begin{equation}\label{eq:product}
(\mu_1\otimes\mu_2)(C_1\times C_2)=\mu_1(C_1)\mu_2(C_2),
\qquad C_s\in\mathcal{R}_s.
\end{equation}
Common finite disjoint refinements in both factors give rectangular grids, so finite additivity of $\mu_1,\mu_2$ makes this prescription well-defined, positive, and left-invariant. For a finite sum $f=\sum_{r,s}c_{rs}\mathbf{1}_{E_r\times F_s}$ on such a grid,
\[
\begin{aligned}
\int_{A_1\times A_2}f\,d(\mu_1\otimes\mu_2)
&=\sum_{r,s}c_{rs}\mu_1(E_r)\mu_2(F_s)\\
&=\int_{A_1}\left(\int_{A_2}f\,d\mu_2\right)d\mu_1
=\int_{A_2}\left(\int_{A_1}f\,d\mu_1\right)d\mu_2.
\end{aligned}
\]
With common component shifts and reference volumes $1$, the cylinder weights multiply according to $d_1+d_2$. Thus the product measure restricts to the measure of the product chart on its congruence cylinder ring.

\begin{proposition}\label{prop:quotient}
Let $\pi:A\to Q$ be an analytic homomorphism of finite-dimensional $F$-analytic groups with surjective derivative at the identity. Put $N=\ker\pi$, $d_N=\dim N$, $d_A=\dim A$, and $d_Q=\dim Q$. There are identity charts $z_N,z_A,z_Q$ and common integers $\kappa_1,\ldots,\kappa_{n-1}\ge0$ and $\kappa_n\ge1$ such that, for every depth
\[
(i_1,\ldots,i_n)\ge(\kappa_1,\ldots,\kappa_n),
\]
the subgroups
\[
\begin{aligned}
K^N_{i_1,\ldots,i_n}&:=z_N^{-1}(\mathfrak{p}_{i_1,\ldots,i_n}^{\,d_N}),\\
K^A_{i_1,\ldots,i_n}&:=z_A^{-1}(\mathfrak{p}_{i_1,\ldots,i_n}^{\,d_A}),\\
K^Q_{i_1,\ldots,i_n}&:=z_Q^{-1}(\mathfrak{p}_{i_1,\ldots,i_n}^{\,d_Q}),
\end{aligned}
\]
fit into exact sequences
\begin{equation}\label{eq:exact}
1\longrightarrow K^N_{i_1,\ldots,i_n}\longrightarrow K^A_{i_1,\ldots,i_n}
\xrightarrow{\pi}K^Q_{i_1,\ldots,i_n}\longrightarrow1.
\end{equation}
The dimensions satisfy $d_A=d_N+d_Q$, and the finite inner indices multiply.
\end{proposition}

\begin{proof}
Let $e_A,e_Q$ be the identity elements of $A,Q$. By the implicit function theorem, $N=\pi^{-1}(e_Q)$ is an analytic submanifold near $e_A$, and there is a local section $\sigma:W_Q\to A$ on an identity neighborhood $W_Q\subseteq Q$ with
\[
\begin{aligned}
T_{e_A}N&=\ker\pi'(e_A),\\
\pi\circ\sigma&=\operatorname{id}_{W_Q},\qquad \sigma(e_Q)=e_A,\\
\pi'(e_A)\sigma'(e_Q)&=\operatorname{id}_{T_{e_Q}Q}.
\end{aligned}
\]
Translations by $N$ extend its local analytic structure. On tangent spaces,
\[
T_{e_A}A=\ker\pi'(e_A)\oplus\sigma'(e_Q)(T_{e_Q}Q),
\qquad
d_A=d_N+d_Q.
\]

Choose identity charts $z_N$ on an open neighborhood $W_N\subseteq N$ and $z_Q$ on $W_Q$, shrinking $W_Q$ if necessary. Define
\[
\eta:\pi^{-1}(W_Q)\longrightarrow N,
\qquad \eta(a)=a\sigma(\pi(a))^{-1}.
\]
Indeed,
\[
\pi(\eta(a))
=\pi(a)\bigl(\pi(\sigma(\pi(a)))\bigr)^{-1}
=\pi(a)\pi(a)^{-1}=e_Q.
\]
The group operations, $\pi$, and $\sigma$ are analytic. Since $N$ is an embedded analytic submanifold, $\eta$ is analytic as an $N$-valued map. Therefore
\[
\Omega=\{a\in\pi^{-1}(W_Q):\eta(a)\in W_N\}
\]
is an open identity neighborhood in $A$.

For $(h,y)\in W_N\times W_Q$, put $\Psi(h,y)=h\sigma(y)$. Then
\[
\pi(\Psi(h,y))=y,\qquad
\eta(\Psi(h,y))=h\sigma(y)\sigma(y)^{-1}=h,
\]
so $\Psi$ takes values in $\Omega$. The two compositions are
\[
\begin{aligned}
(\eta,\pi)(\Psi(h,y))&=(h,y),\\
\Psi(\eta(a),\pi(a))
&=a\sigma(\pi(a))^{-1}\sigma(\pi(a))=a.
\end{aligned}
\]
Both maps are analytic, so they are inverse analytic isomorphisms between $W_N\times W_Q$ and $\Omega$. Consequently,
\[
z_A(a)=\bigl(z_N(\eta(a)),z_Q(\pi(a))\bigr)
\]
is an identity chart with values in $F^{d_N}\times F^{d_Q}$. Since $\sigma(e_Q)=e_A$, its restrictions satisfy
\[
\begin{aligned}
z_A(h)&=(z_N(h),0)\qquad(h\in W_N),\\
z_Q(\pi(a))&=\operatorname{pr}_2(z_A(a)),\\
a\in N&\ \Longleftrightarrow\ \operatorname{pr}_2(z_A(a))=0,
\end{aligned}
\]
where $\operatorname{pr}_2$ is projection onto the second coordinate factor.

Apply Lemma~\ref{lem:estimate} to the positive-dimensional factors and choose common component bounds $\kappa_1,\ldots,\kappa_{n-1}\ge0$ and $\kappa_n\ge1$. Enlarge them so that the balls of depth $(\kappa_1,\ldots,\kappa_n)$ lie in the three chart domains. For every $(i_1,\ldots,i_n)\ge(\kappa_1,\ldots,\kappa_n)$, nesting places the corresponding balls in those domains, and Proposition~\ref{prop:deep} gives subgroups with
\[
\begin{aligned}
\mathfrak{p}_{i_1,\ldots,i_n}^{\,d_A}
&=\mathfrak{p}_{i_1,\ldots,i_n}^{\,d_N}\times\mathfrak{p}_{i_1,\ldots,i_n}^{\,d_Q},\\
K^A_{i_1,\ldots,i_n}
&=\{h\sigma(y):h\in K^N_{i_1,\ldots,i_n},\
y\in K^Q_{i_1,\ldots,i_n}\}.
\end{aligned}
\]
Consequently,
\[
\begin{aligned}
\pi(K^A_{i_1,\ldots,i_n})&=K^Q_{i_1,\ldots,i_n},\\
\ker(\pi|_{K^A_{i_1,\ldots,i_n}})
&=K^A_{i_1,\ldots,i_n}\cap N
=K^N_{i_1,\ldots,i_n}.
\end{aligned}
\]
Together with injectivity of the inclusion, these identities prove \eqref{eq:exact}.

Finally, for $m\ge0$,
\[
\begin{aligned}
[K^A_{i_1,\ldots,i_n}:K^A_{i_1+m,i_2,\ldots,i_n}]
&=q^{d_Am}=q^{d_Nm}q^{d_Qm}\\
&=[K^N_{i_1,\ldots,i_n}:K^N_{i_1+m,i_2,\ldots,i_n}]\\
&\qquad{}\cdot[K^Q_{i_1,\ldots,i_n}:K^Q_{i_1+m,i_2,\ldots,i_n}].
\end{aligned}
\]
For a zero-dimensional factor, $F^0=\{0\}$, its coordinate subgroup is the identity singleton, and its index is $q^{0m}=1$.
\end{proof}

Choose positive constants $c_N,c_A,c_Q\in\mathbb R_n$ with $c_A=c_Nc_Q$, and set
\[
\mu_N=c_N\mu^{\mathrm{abs}}_{N,z_N},\qquad
\mu_A=c_A\mu^{\mathrm{abs}}_{A,z_A},\qquad
\mu_Q=c_Q\mu^{\mathrm{abs}}_{Q,z_Q},
\]
using the corresponding constant times counting measure for any zero-dimensional factor. The dimension identity and the absolute cylinder volumes then give
\begin{equation}\label{eq:quotientnormalization}
\mu_A(K^A_{i_1,\ldots,i_n})=\mu_N(K^N_{i_1,\ldots,i_n})\mu_Q(K^Q_{i_1,\ldots,i_n}).
\end{equation}
Smooth quotient morphisms of algebraic groups satisfy the hypotheses above.

\begin{proposition}\label{prop:disintegration}
For the families and normalizations above, let $f$ be a finite linear combination of indicators $\mathbf{1}_{aK^A_{i_1,\ldots,i_n}}$ with $(i_1,\ldots,i_n)\ge(\kappa_1,\ldots,\kappa_n)$. For $y\in\pi(A)$, choose $a_y\in A$ with $\pi(a_y)=y$ and set
\[
(\pi_!f)(y)=\int_N f(a_yh)\,d\mu_N(h),
\]
with value zero outside $\pi(A)$. This is independent of the lift, is cylinder-simple on $Q$, and satisfies
\begin{equation}\label{eq:disintegration}
\int_A f\,d\mu_A=\int_Q\pi_!f\,d\mu_Q.
\end{equation}
\end{proposition}

\begin{proof}
Take $f=\mathbf{1}_{aK^A_{i_1,\ldots,i_n}}$. Exactness gives
\[
a_yh\in aK^A_{i_1,\ldots,i_n}\ (h\in N)
\ \Longrightarrow\
y\in\pi(a)K^Q_{i_1,\ldots,i_n}.
\]
Conversely, for $y$ in this last cylinder, choose
\[
k_0\in K^A_{i_1,\ldots,i_n},\qquad
\pi(k_0)=\pi(a)^{-1}y.
\]
Then
\[
h_0=a_y^{-1}ak_0\in N,\qquad a_yh_0=ak_0,
\]
and
\[
\begin{aligned}
a_yh\in aK^A_{i_1,\ldots,i_n}
&\ \Longleftrightarrow\
(a_yh_0)^{-1}(a_yh)\in K^A_{i_1,\ldots,i_n}\\
&\ \Longleftrightarrow\
h_0^{-1}h\in K^A_{i_1,\ldots,i_n}\cap N\\
&\ \Longleftrightarrow\
h\in h_0K^N_{i_1,\ldots,i_n}.
\end{aligned}
\]
Left invariance on $N$ therefore yields
\[
\pi_!\mathbf{1}_{aK^A_{i_1,\ldots,i_n}}
=\mu_N(K^N_{i_1,\ldots,i_n})
 \mathbf{1}_{\pi(a)K^Q_{i_1,\ldots,i_n}}.
\]
Its support lies in $\pi(A)$, and it is cylinder-simple.

For a different lift, write $a_y'=a_yn_0$ with $n_0\in N$. The substitution $h'=n_0h$ gives
\[
\int_N f(a_y'h)\,d\mu_N(h)
=\int_N f(a_yn_0h)\,d\mu_N(h)
=\int_N f(a_yh')\,d\mu_N(h').
\]
This proves independence of the lift. Finite linearity gives
\[
\pi_!\left(\sum_{r=1}^M c_rf_r\right)
=\sum_{r=1}^M c_r\pi_!f_r.
\]
The integral on $N$ depends on the function itself, so different presentations of $f$ give the same pushforward.

For a cylinder indicator, \eqref{eq:quotientnormalization} gives
\[
\begin{aligned}
\int_Q\pi_!\mathbf{1}_{aK^A_{i_1,\ldots,i_n}}\,d\mu_Q
&=\mu_N(K^N_{i_1,\ldots,i_n})\mu_Q(K^Q_{i_1,\ldots,i_n})\\
&=\mu_A(K^A_{i_1,\ldots,i_n})\\
&=\int_A\mathbf{1}_{aK^A_{i_1,\ldots,i_n}}\,d\mu_A.
\end{aligned}
\]
Summation over the finitely many terms of $f$ proves \eqref{eq:disintegration}.
\end{proof}

\section{Examples}\label{sec:examples}

Let $\mu_G$ be the measure with absolute normalization in the coordinates used below. For $d=\dim G$, its cylinder volumes are $\mu_G(K_{i_1,\ldots,i_n})=q^{-di_1}\prod_{\ell=2}^nX_\ell^{di_\ell}$. For the component shifts $\kappa_1,\ldots,\kappa_n$, the normalization of Theorem~\ref{thm:main} is
\[
\mu^0_{G,z}
=q^{d\kappa_1}\left(\prod_{\ell=2}^nX_\ell^{-d\kappa_\ell}\right)\mu_G.
\]
For additive groups the formula holds at every depth.

The explicit coordinates for $\operatorname{GL}_m$, $\operatorname{SL}_m$ ($m\ge2$), the Heisenberg group, and the affine group below also give subgroups at every depth with $i_n>0$. Write $K_{i_1,\ldots,i_n}$ for any of these coordinate subgroups and $d$ for its dimension. The nonlinear terms are integral in these coordinates, so
\[
\begin{gathered}
\mathfrak{p}_{i_1,\ldots,i_n}\mathfrak{p}_{i_1,\ldots,i_n}\subseteq\mathfrak{p}_{i_1+1,i_2,\ldots,i_n},\\
K_{i_1,\ldots,i_n}/K_{i_1+1,i_2,\ldots,i_n}\simeq(k^d,+).
\end{gathered}
\]
The index argument of Proposition~\ref{prop:deep} therefore applies throughout $J=\{(i_1,\ldots,i_n):i_n>0\}$. Proposition~\ref{prop:extension}, with $B=F_1^d$ and $U_{i_1}=\pi_1^{i_1}\mathcal{O}_{F_1}^d$, supplies the measures on these larger cylinder rings after substituting $Y_\ell=X_\ell^d$. Their restrictions to sufficiently deep tuple ranges agree with the absolute normalization by uniqueness.

\begin{example}\label{ex:equal-characteristic}\label{ex:mixed-characteristic}
\emph{(i) Equal characteristic.} Let
\[
F=k((t_1))\cdots((t_n)),\qquad |k|=q.
\]
With $F_{n-1}=k((t_1))\cdots((t_{n-1}))$, the balls have the recursive description
\[
\mathfrak{p}_{i_1,\ldots,i_n}
=t_n^{i_n}\mathfrak{p}_{i_1,\ldots,i_{n-1}}(F_{n-1})
+t_n^{i_n+1}F_{n-1}[[t_n]],
\]
starting from $\mathfrak{p}_{i_1}(F_1)=t_1^{i_1}k[[t_1]]$. Thus, for $F=k((u))((t))((s))$,
\[
\mathfrak{p}_{i_1,i_2,i_3}
=u^{i_1}t^{i_2}s^{i_3}k[[u]]
+t^{i_2+1}s^{i_3}k((u))[[t]]
+s^{i_3+1}k((u))((t))[[s]].
\]
The additive measure satisfies
\[
\mu_{F^d}(a+\mathfrak{p}_{i_1,\ldots,i_n}^{\,d})=q^{-di_1}\prod_{\ell=2}^nX_\ell^{di_\ell}.
\]
For $d=1$, choosing the leading coefficient in $k$ gives the disjoint partition
\[
\mathfrak{p}_{i_1,\ldots,i_n}
=\bigsqcup_{\alpha\in k}
\bigl(\alpha t_1^{i_1}\cdots t_n^{i_n}+\mathfrak{p}_{i_1+1,i_2,\ldots,i_n}\bigr),
\]
and hence
\[
\mu_F(\mathfrak{p}_{i_1,\ldots,i_n}\setminus\mathfrak{p}_{i_1+1,i_2,\ldots,i_n})
=(1-q^{-1})q^{-i_1}\prod_{\ell=2}^nX_\ell^{i_\ell}.
\]
For $2\le\ell\le n$, an outer-component shell instead has volume
\[
\mu_F(\mathfrak{p}_{i_1,\ldots,i_n}\setminus\mathfrak{p}_{i_1,\ldots,i_\ell+1,\ldots,i_n})
=q^{-i_1}\Bigl(\prod_{s=2}^nX_s^{i_s}\Bigr)(1-X_\ell)>0.
\]
In dimension three, for every integer $a\ge0$,
\begin{equation}\label{eq:evaluation}
\mu_F(\mathfrak{p}_{a,0,0}\setminus\mathfrak{p}_{0,1,0})=q^{-a}-X_2>0,
\qquad
\mu_F(\mathfrak{p}_{0,a,0}\setminus\mathfrak{p}_{0,0,1})=X_2^a-X_3>0.
\end{equation}
The scalar-ball volumes agree with Fesenko's normalization \cite{fesenko,fesenko-higher} in dimensions two and higher.

\emph{(ii) Mixed characteristic.} Let $K/\mathbb Q_p$ be finite, with uniformizer $\pi$ and residue field $k$ of order $q$. For
\[
F=K((t_2))\cdots((t_n))
\]
the parameter sequence is $(\pi,t_2,\ldots,t_n)$ and $F_1=K$, so
\[
\mu_F(\pi^{i_1}t_2^{i_2}\cdots t_n^{i_n}\mathcal{O})
=q^{-i_1}X_2^{i_2}\cdots X_n^{i_n}.
\]

For comparison, let $K\{\{u\}\}$ consist of series $\sum_{s\in\mathbb{Z}}a_su^s$ such that the valuations of the $a_s\in K$ are bounded below and tend to $+\infty$ as $s\to-\infty$. Its discrete valuation is the minimum coefficient valuation, and its residue field is $k((u))$ \cite{zhukov,morrow-higher}. The field
\[
F=K\{\{u\}\}((t_3))\cdots((t_n))
\]
has the parameter sequence $(u,\pi,t_3,\ldots,t_n)$. Consequently,
\[
\mu_F(u^{i_1}\pi^{i_2}t_3^{i_3}\cdots t_n^{i_n}\mathcal{O})
=q^{-i_1}X_2^{i_2}\cdots X_n^{i_n}.
\]

\end{example}

\begin{example}\label{ex:matrix-groups}
Write $I_m$ for the identity matrix and $M_m(S)$ for the $m\times m$ matrices over $S$. For $\operatorname{GL}_m(F)$, put $C_{i_1,\ldots,i_n}=I_m+M_m(\mathfrak{p}_{i_1,\ldots,i_n})$ at positive outer depth. Matrix entries give $m^2$ identity coordinates, so
\[
\begin{gathered}
\mu_{\operatorname{GL}_m}(gC_{i_1,\ldots,i_n})=q^{-m^2i_1}\prod_{\ell=2}^nX_\ell^{m^2i_\ell},
\\
[C_{i_1,\ldots,i_n}:C_{i_1+a,i_2,\ldots,i_n}]=q^{m^2a}\quad(a\ge0).
\end{gathered}
\]
For $n=2$ and $m=2$, let $\nu_2^{\mathrm W}$ denote the $\mathbb{R}((X))$-valued, finitely additive, left-invariant measure constructed by Waller~\cite[Theorem~4.3]{waller-gl2} on the ring generated by his congruence cosets in $\operatorname{GL}_2(F)$. The superscript $\mathrm W$ stands for Waller, and the subscript records the matrix size.

His subgroups are $K_{i,j}=C_{i,j}$, and his normalization gives
\[
\nu_2^{\mathrm W}(\operatorname{GL}_2(\mathcal{O}))=1,
\qquad
\nu_2^{\mathrm W}(gC_{i,j})
=\frac{q^3}{(q^2-1)(q-1)}q^{-4i}X^{4j}.
\]
On the cylinder ring with $j>0$, identifying $X=X_2$ therefore gives
\[
\mu_{\operatorname{GL}_2}
=\frac{(q^2-1)(q-1)}{q^3}\nu_2^{\mathrm W}.
\]
For arbitrary $m$, let $\nu_m^{\mathrm W}$ denote Waller's $\mathbb{R}((Y_2))\cdots((Y_n))$-valued, finitely additive, left-invariant measure on the ring generated by congruence cosets in $\operatorname{GL}_m(F)$, constructed in~\cite[Section~4, pp.~19--20]{waller}. Here $m$ is the matrix size and $n$ is the field dimension; $\mu_{\operatorname{GL}_m}$ retains our absolute coordinate normalization.

The normalization of $\nu_m^{\mathrm W}$ is induced by Haar measure $h_m$ on $\operatorname{GL}_m(F_1)$ with
\[
h_m\bigl(\operatorname{GL}_m(\mathcal{O}_{F_1})\bigr)=1,
\qquad
U_a=I_m+\pi_1^aM_m(\mathcal{O}_{F_1})\quad(a\ge1).
\]
Reduction modulo $\pi_1$ and the successive congruence quotients give
\[
\begin{aligned}
\#\operatorname{GL}_m(k)
&=\prod_{r=0}^{m-1}(q^m-q^r)
=q^{m(m-1)/2}\prod_{r=1}^m(q^r-1),\\
[\operatorname{GL}_m(\mathcal{O}_{F_1}):U_a]
&=\#\operatorname{GL}_m(k)\,q^{m^2(a-1)},\\
h_m(U_a)
&=c_mq^{-m^2a},
\end{aligned}
\]
where
\[
c_m=\frac{q^{m^2}}{\#\operatorname{GL}_m(k)}
=\frac{q^{m(m+1)/2}}{\prod_{r=1}^m(q^r-1)}
=\prod_{r=1}^m(1-q^{-r})^{-1}.
\]
Writing $Y_2,\ldots,Y_n$ for Waller's level variables, his cylinder formula becomes
\[
\nu_m^{\mathrm W}(gC_{i_1,\ldots,i_n})
=c_mq^{-m^2i_1}\prod_{\ell=2}^nY_\ell^{i_\ell}
\qquad(i_n>0).
\]
Consequently,
\[
\begin{gathered}
Y_\ell\longmapsto X_\ell^{m^2}\quad(2\le\ell\le n),\\
c_m^{-1}\left.\nu_m^{\mathrm W}(S)\right|_{Y_\ell=X_\ell^{m^2}}
=\mu_{\operatorname{GL}_m}(S).
\end{gathered}
\]
Here $S$ lies in the ring generated by the common cylinders of positive outer depth; finite additivity extends the equality from cylinders to this ring.

For $m\ge2$, the group $\operatorname{SL}_m(F)=\{g\in\operatorname{GL}_m(F):\det g=1\}$ has the following local parametrization in every residue characteristic:
\[
g=\begin{pmatrix}B&\beta\\ \eta&d_0\end{pmatrix},
\qquad
d_0=(\det B)^{-1}+\eta B^{-1}\beta,
\]
where $B\in\operatorname{GL}_{m-1}(F)$, $\beta$ is a column, and $\eta$ is a row. Thus $B-I_{m-1},\beta,\eta$ give $m^2-1$ coordinates, and
\begin{equation}\label{eq:SLvolume}
\mu_{\operatorname{SL}_m}\bigl(g(C_{i_1,\ldots,i_n}\cap\operatorname{SL}_m(F))\bigr)
=q^{-(m^2-1)i_1}\prod_{\ell=2}^nX_\ell^{(m^2-1)i_\ell}.
\end{equation}
The group $\operatorname{SL}_1(F)=\{1\}$ has counting measure. The determinant sequence
\[
1\longrightarrow C_{i_1,\ldots,i_n}\cap\operatorname{SL}_m(F)
\longrightarrow C_{i_1,\ldots,i_n}
\xrightarrow{\det}1+\mathfrak{p}_{i_1,\ldots,i_n}\longrightarrow1
\]
has a section $a\mapsto\operatorname{diag}(a,1,\ldots,1)$. Write $\det_!$ for the fiber-integration operator $\pi_!$ of Proposition~\ref{prop:disintegration} with $\pi=\det$. The volumes multiply, giving
\[
\det_!\mathbf{1}_{gC_{i_1,\ldots,i_n}}
=q^{-(m^2-1)i_1}\prod_{\ell=2}^nX_\ell^{(m^2-1)i_\ell}
\mathbf{1}_{\det(g)(1+\mathfrak{p}_{i_1,\ldots,i_n})}.
\]
Positive outer depth allows negative inner components. For instance, over $k((t_1))((t_2))((t_3))((t_4))$ at depth $(1,-2,0,1)$,
\[
\mu_{\operatorname{GL}_2}(C_{1,-2,0,1})=q^{-4}X_2^{-8}X_4^4,
\qquad
\mu_{\operatorname{SL}_2}(C_{1,-2,0,1}\cap\operatorname{SL}_2(F))=q^{-3}X_2^{-6}X_4^3.
\]

\end{example}

The following lemma and proposition give a common description of sufficiently deep matrix congruences for smooth closed subgroups of $\operatorname{GL}_m$.

\begin{lemma}\label{lem:tangent}
Let $L:F^d\hookrightarrow F^N$ be injective and $F$-linear. There are $M\in\operatorname{GL}_d(F)$ and $U\in\operatorname{GL}_N(\mathcal{O})$ such that
\[
ULM=\begin{pmatrix}I_d\\0\end{pmatrix}.
\]
Both $U$ and $U^{-1}$ preserve every $\mathfrak{p}_{i_1,\ldots,i_n}^{\,N}$.
\end{lemma}

\begin{proof}
For a set $I$ of $d$ row indices, let $L_I$ be the corresponding square submatrix. Since $\operatorname{rank}L=d$, choose a row permutation $P$ such that
\[
PL=\begin{pmatrix}B\\D\end{pmatrix},\qquad
v(\det B)=\min_{\det L_I\ne0}v(\det L_I).
\]
Put $M=B^{-1}$ and $C=DB^{-1}$. If $B[r\leftarrow D_s]$ is obtained by replacing row $r$ by row $s$ of $D$, Cramer's rule gives
\[
C_{sr}=\frac{\det B[r\leftarrow D_s]}{\det B},
\qquad
v(\det B[r\leftarrow D_s])\ge v(\det B)
\ \Longrightarrow\ C_{sr}\in\mathcal{O}.
\]
A zero numerator is covered by $v(0)=+\infty$. Now set
\[
U=\begin{pmatrix}I_d&0\\-C&I_{N-d}\end{pmatrix}P,\qquad
U^{-1}=P^{-1}\begin{pmatrix}I_d&0\\C&I_{N-d}\end{pmatrix}.
\]
Both matrices are integral, and
\[
ULM=
\begin{pmatrix}I_d&0\\-C&I_{N-d}\end{pmatrix}
\begin{pmatrix}B\\D\end{pmatrix}B^{-1}
=\begin{pmatrix}I_d\\-C+DB^{-1}\end{pmatrix}
=\begin{pmatrix}I_d\\0\end{pmatrix}.
\]
For every depth, the $\mathcal{O}$-module property yields
\[
U^{\pm1}\mathfrak{p}_{i_1,\ldots,i_n}^{\,N}
\subseteq\mathfrak{p}_{i_1,\ldots,i_n}^{\,N}
\ \Longrightarrow\
U\mathfrak{p}_{i_1,\ldots,i_n}^{\,N}=\mathfrak{p}_{i_1,\ldots,i_n}^{\,N}.
\]
In particular,
\[
v(Ux)\ge v(x)=v(U^{-1}Ux)\ge v(Ux),
\qquad v(Ux)=v(x).
\]
\end{proof}

\begin{proposition}\label{prop:matrix}
Let $H\hookrightarrow\operatorname{GL}_m$ be a smooth closed algebraic $F$-subgroup of dimension $d>0$, and set
\[
C^H_{i_1,\ldots,i_n}=H(F)\cap\bigl(I_m+M_m(\mathfrak{p}_{i_1,\ldots,i_n})\bigr).
\]
There are identity coordinates $z$, integers $\kappa_1,\ldots,\kappa_{n-1}\ge0$, and an integer $\kappa_n\ge1$ such that, for every depth $(i_1,\ldots,i_n)\in\mathbb{Z}^n$ satisfying
\[
(i_1,\ldots,i_n)\ge(\kappa_1,\ldots,\kappa_n)
\quad\text{in the right lexicographic order},
\]
one has
\begin{equation}\label{eq:matrixidentity}
C^H_{i_1,\ldots,i_n}=z^{-1}(\mathfrak{p}_{i_1,\ldots,i_n}^{\,d}).
\end{equation}
The left translates of these subgroups generate a ring with a left-invariant measure
\begin{equation}\label{eq:matrixvolume}
\mu_H(hC^H_{i_1,\ldots,i_n})=c\,q^{-di_1}\prod_{\ell=2}^nX_\ell^{di_\ell},
\qquad c\in(\mathbb R_n)_{>0}.
\end{equation}
For every such depth and every integer $m_0\ge0$, the index $[C^H_{i_1,\ldots,i_n}:C^H_{i_1+m_0,i_2,\ldots,i_n}]$ is $q^{dm_0}$.
\end{proposition}

\begin{proof}
Identify $M_m(F)$ with $F^{m^2}$. Choose an analytic parametrization $\varphi$ of $H$ at $I_m$. Smoothness and the closed immersion give
\[
\varphi(0)=I_m,\qquad \operatorname{rank}\varphi'(0)=d.
\]
By Lemma~\ref{lem:tangent}, choose $M\in\operatorname{GL}_d(F)$ and $U\in\operatorname{GL}_{m^2}(\mathcal{O})$ with
\[
U\varphi'(0)M=\begin{pmatrix}I_d\\0\end{pmatrix}.
\]
For $\iota_H(x)=\varphi(Mx)$, the Taylor expansion becomes
\[
U(\iota_H(x)-I_m)=\begin{pmatrix}x+N_1(x)\\N_2(x)\end{pmatrix},
\qquad
N_1(0)=N_2(0)=0,\quad N_1'(0)=N_2'(0)=0.
\]

Factor one coordinate out of each remainder monomial:
\[
(N_1,N_2)_s(x)=\sum_{r=1}^d x_rT_{sr}(x),
\qquad T_{sr}(0)=0.
\]
Shrink the domain until $T_{sr}(x)\in t_n\mathcal{O}_F$. Then, for $x\ne0$,
\[
\begin{aligned}
v(T_{sr}(x))&>0,\\
v(N_1(x)),\ v(N_2(x))
&\ge\min_{s,r}\{v(x_r)+v(T_{sr}(x))\}>v(x),\\
v(x+N_1(x))&=v(x),\\
v\bigl((x+N_1(x),N_2(x))\bigr)&=v(x).
\end{aligned}
\]
Zero terms have valuation $+\infty$; an empty block is omitted. Since $U$ preserves the minimum coordinate valuation,
\begin{equation}\label{eq:minvaluation}
\min_{s,t}v\bigl(\iota_H(x)_{st}-(I_m)_{st}\bigr)=v(x).
\end{equation}
The equality includes $x=0$.

Choose the chart with $\iota_H(W)=H(F)\cap\Omega$, where $\Omega$ is an ambient identity neighborhood. Choose component constants as in Lemma~\ref{lem:estimate}, enlarged so that
\[
\mathfrak{p}_{\kappa_1,\ldots,\kappa_n}^{\,d}\subset W,\qquad
I_m+M_m(\mathfrak{p}_{\kappa_1,\ldots,\kappa_n})\subset\Omega.
\]
For $(i_1,\ldots,i_n)\ge(\kappa_1,\ldots,\kappa_n)$, nesting places the corresponding coordinate ball in $W$ and the matrix ball in $\Omega$. For $z=\iota_H^{-1}$ and $g\in H(F)\cap\Omega$, equation \eqref{eq:minvaluation} gives
\[
\begin{aligned}
g\in I_m+M_m(\mathfrak{p}_{i_1,\ldots,i_n})
&\ \Longleftrightarrow\
\min_{s,t}v(g_{st}-(I_m)_{st})\ge(i_1,\ldots,i_n)\\
&\ \Longleftrightarrow\
v(z(g))\ge(i_1,\ldots,i_n)\\
&\ \Longleftrightarrow\
z(g)\in\mathfrak{p}_{i_1,\ldots,i_n}^{\,d}.
\end{aligned}
\]
This proves \eqref{eq:matrixidentity} for every depth in the statement. For $m_0\ge0$,
\[
(i_1+m_0,i_2,\ldots,i_n)
\ge(i_1,\ldots,i_n)\ge(\kappa_1,\ldots,\kappa_n),
\]
so both depths remain in this family. Proposition~\ref{prop:deep} therefore gives subgroups with
\[
[C^H_{i_1,\ldots,i_n}:C^H_{i_1+m_0,i_2,\ldots,i_n}]=q^{dm_0}.
\]
The analytic-group construction in the proof of Theorem~\ref{thm:main}, with these component shifts and \eqref{eq:absnormalization}, gives the absolute cylinder measure on the ring generated by this family. For any fixed $c\in(\mathbb R_n)_{>0}$, the left-invariant measure $\mu_H=c\,\mu^{\mathrm{abs}}_{H(F),z}$ satisfies
\[
\begin{aligned}
\mu_H(hC^H_{i_1,\ldots,i_n})
&=c\,\mu^{\mathrm{abs}}_{H(F),z}(hC^H_{i_1,\ldots,i_n})\\
&=c\,q^{-di_1}\prod_{\ell=2}^nX_\ell^{di_\ell},
\end{aligned}
\]
as required.
\end{proof}

For the symplectic group $H=\operatorname{Sp}_{2m}$, use the adapted identity coordinates and component constants of Proposition~\ref{prop:matrix}. Since $\dim H=m(2m+1)$, the matrix congruence subgroups at depths $(i_1,\ldots,i_n)\ge(\kappa_1,\ldots,\kappa_n)$ have absolute volume
\[
\mu_H(gC^H_{i_1,\ldots,i_n})
=q^{-m(2m+1)i_1}\prod_{\ell=2}^nX_\ell^{m(2m+1)i_\ell}.
\]

\begin{example}\label{ex:norm-one-torus}
Assume $\operatorname{char}F\ne2$ and choose a nonsquare $D\in F^\times$. The norm-one torus for $F(\sqrt D)/F$ has the realization
\[
T(F)=\left\{\begin{pmatrix}a&Dx\\x&a\end{pmatrix}:a^2-Dx^2=1\right\}.
\]
The derivative with respect to $a$ is $2$ at $(a,x)=(1,0)$, so there is a unique analytic solution $a=a(x)$ near the identity. Since
\[
a(x)-1=\frac{Dx^2}{a(x)+1},
\]
$x$ is an identity coordinate. Choose its component constants as in Proposition~\ref{prop:deep}. For $(i_1,\ldots,i_n)\ge(\kappa_1,\ldots,\kappa_n)$,
\[
K^T_{i_1,\ldots,i_n}=\left\{\begin{pmatrix}a(x)&Dx\\x&a(x)\end{pmatrix}:x\in\mathfrak{p}_{i_1,\ldots,i_n}\right\},
\qquad
\mu_T(gK^T_{i_1,\ldots,i_n})=q^{-i_1}\prod_{\ell=2}^nX_\ell^{i_\ell}.
\]
For instance, take $D=u$ over $F=\mathbb F_q((u))((t))((s))$ with $q$ odd. Since $T$ is commutative, its measure is bi-invariant.

\end{example}

\begin{example}\label{ex:heisenberg}
Write the Heisenberg group as
\[
h(x,y,w)=\begin{pmatrix}1&x&w\\0&1&y\\0&0&1\end{pmatrix},
\qquad \mathcal H(F)=\{h(x,y,w):x,y,w\in F\}.
\]
Its multiplication is
\[
h(x,y,w)h(x',y',w')=h(x+x',y+y',w+w'+xy').
\]
Inversion and the commutator are
\[
\begin{aligned}
h(x,y,w)^{-1}&=h(-x,-y,xy-w),\\
[h(x,y,w),h(x',y',w')]&=h(0,0,xy'-x'y).
\end{aligned}
\]
For $i_n>0$, the set $K^{\mathcal H}_{i_1,\ldots,i_n}=\{h(x,y,w):x,y,w\in\mathfrak{p}_{i_1,\ldots,i_n}\}$ is a subgroup because $\mathfrak{p}_{2i_1,\ldots,2i_n}\subset\mathfrak{p}_{i_1,\ldots,i_n}$. The bilinear commutator gives an explicit instance of Lemma~\ref{lem:estimate}:
\[
[K^{\mathcal H}_{i_1,\ldots,i_n},K^{\mathcal H}_{j_1,\ldots,j_n}]
\subseteq K^{\mathcal H}_{i_1+j_1,\ldots,i_n+j_n}
\qquad(i_n,j_n>0).
\]
The center $Z=\{h(0,0,w):w\in F\}$ is the kernel of the quotient map $\pi:\mathcal H(F)\to F^2$, $\pi(h(x,y,w))=(x,y)$. In absolute normalization, the volume and fiber integral are
\[
\begin{gathered}
\mu_{\mathcal H}(K^{\mathcal H}_{i_1,\ldots,i_n})=q^{-3i_1}\prod_{\ell=2}^nX_\ell^{3i_\ell},
\\
\pi_!\mathbf{1}_{K^{\mathcal H}_{i_1,\ldots,i_n}}
=q^{-i_1}\prod_{\ell=2}^nX_\ell^{i_\ell}\mathbf{1}_{\mathfrak{p}_{i_1,\ldots,i_n}^{\,2}},
\end{gathered}
\]
as in Proposition~\ref{prop:disintegration}.

\end{example}

\begin{example}\label{ex:affine-group}
The affine group is
\[
A=F^\times\ltimes F,\qquad (a,b)(a',b')=(aa',b+ab').
\]
At positive outer depth, the coordinates $(a-1,b)$ give
\[
K^A_{i_1,\ldots,i_n}=\{(a,b):a-1,b\in\mathfrak{p}_{i_1,\ldots,i_n}\},
\qquad \mu_A(K^A_{i_1,\ldots,i_n})=q^{-2i_1}\prod_{\ell=2}^nX_\ell^{2i_\ell}.
\]
A left-invariant form is $\omega=a^{-2}\,da\wedge db$. Since $\delta_A(a,b)=a$, right translation by $g_0=(a_0,b_0)$ gives
\[
\lambda_\omega(Sg_0)=\left|a_0\right|_n^{-1}\lambda_\omega(S).
\]
In particular, if $a_0=t_1^{j_1}\cdots t_n^{j_n}$, the factor is
\[
q^{j_1}\prod_{\ell=2}^nX_\ell^{-j_\ell}.
\]

\end{example}

\begin{example}\label{ex:abelian-varieties}
Let $\mathcal A/F$ be an abelian variety of dimension $h>0$. Choose identity coordinates $z_{\mathcal A}$ with values in $F^h$ and component constants as in Proposition~\ref{prop:deep}. For $(i_1,\ldots,i_n)\ge(\kappa_1,\ldots,\kappa_n)$, set
\[
K^{\mathcal A}_{i_1,\ldots,i_n}
:=z_{\mathcal A}^{-1}(\mathfrak{p}_{i_1,\ldots,i_n}^{\,h}).
\]
For $P\in\mathcal A(F)$ and $m\ge0$, this gives
\[
\mu_{\mathcal A}(P+K^{\mathcal A}_{i_1,\ldots,i_n})=q^{-hi_1}\prod_{\ell=2}^nX_\ell^{hi_\ell},
\qquad
[K^{\mathcal A}_{i_1,\ldots,i_n}:K^{\mathcal A}_{i_1+m,i_2,\ldots,i_n}]=q^{hm}.
\]
The measure is bi-invariant because $\mathcal A(F)$ is commutative. For an elliptic curve $\mathcal E/F$ with identity $O$, the formal group gives an identity coordinate $z$ with $z(O)=0$ \cite[Chapter~IV]{silverman}; the formulas above apply with $h=1$. These give non-affine instances of Theorem~\ref{thm:main}.

\end{example}

\par\medskip\noindent\textbf{Acknowledgements:}
The authors thank Ivan Fesenko for his support and discussions on higher local fields. This work was supported by Westlake University grants 207010146022301 and 207010145609901, and by the Hangzhou Postdoctoral Daily Funding program.

\Needspace{12\baselineskip}

\end{document}